\documentclass{article}[10pt]

\usepackage[section]{placeins}
\usepackage[titletoc,title]{appendix}

\usepackage{amsmath}
\usepackage{amssymb}
\usepackage{amstext}
\usepackage{amsthm}
\usepackage{mathtools}
\usepackage{tensor}
\usepackage{bbm}
\usepackage{bbold}

\usepackage{pdflscape}
\usepackage{lscape}
\usepackage{relsize}
\allowdisplaybreaks

\usepackage[colorlinks,%
citecolor=blue]{hyperref}

\usepackage[all]{xy}

\newcommand{\C}{\mathcal C}
\newcommand{\D}{\mathcal D}

\newcommand{\op}{\mathsf{op}}
\newcommand{\cop}{\mathsf{cop}}

\newcommand{\I}{\mathbbm 1}

\newcommand{\ot}{{\otimes}}
\newcommand{\pl}{{\oplus}}

\newcommand{\eps}{\epsilon}

\newcommand{\e}{\mathsf{e}}
\newcommand{\db}{\mathsf{d}}

\newcommand{\hol}{\mathbbm H^l} % left Hopf operator
\newcommand{\hor}{\mathbbm H^r} % right Hopf operator

\newcommand{\chol}{\mathfrak H^l} % left coHopf operator
\newcommand{\chor}{\mathfrak H^r} % right coHopf operator

\newcommand{\bhol}{\overline{\mathbbm H}^l} % left fusion operator
\newcommand{\bhor}{\overline{\mathbbm H}^r} % right fusion operator

\newcommand{\nrr}{\nu^r_R} % right costrength
\newcommand{\nlr}{\nu^l_R} % left costrength

\newcommand{\nrl}{\nu^r_L} % right strength
\newcommand{\nll}{\nu^l_L} % left strength

\newcommand{\rotatedadj}{\rotatebox[origin=c]{63}{\ensuremath\vdash}} % rotation of the adjunction symbol

\newtheorem{theorem}{Theorem}
\newtheorem{definition}[theorem]{Definition}
\newtheorem{proposition}[theorem]{Proposition}
\newtheorem{lemma}[theorem]{Lemma}
\newtheorem{remark}[theorem]{Remark}
\newtheorem{corollary}[theorem]{Corollary}
\newtheorem{example}[theorem]{Example}

\begin{document}

\title{On Hopf adjunctions, Hopf monads and Frobenius-type properties\thanks{
This work was supported by a grant of the Romanian National Authority for Scientific Research, CNCS-UEFISCDI, project number PN-II-RU-TE-2012-3-0168.}}

\author{Adriana Balan\thanks{%
Simion Stoilow Institute of Mathematics of the Romanian Academy, Bucharest, Romania, and University Politehnica of Bucharest, Romania}}

\date{}

\maketitle

%===================================================================%

\begin{abstract}
Let $U$ be a strong monoidal functor between monoidal categories. If it has both a left adjoint $L$ and a right adjoint $R$, we show that the pair $(R,L)$ is a linearly distributive functor  and $(U,U)\dashv (R,L)$ is a linearly distributive adjunction, if and only if $L\dashv U$ is a Hopf adjunction and $U\dashv R$ is a coHopf adjunction. 

We give sufficient conditions for a strong monoidal $U$ which is part of a (left) Hopf adjunction $L\dashv U$, to have as right adjoint a twisted version of the left adjoint $L$. In particular, the resulting adjunction will be (left) coHopf. One step further, we prove that if $L$ is precomonadic and $L\I$ is a Frobenius monoid (where $\I$ denotes the unit object of the monoidal category), then $L\dashv U\dashv L$ is an ambidextrous adjunction, and $L$ is a Frobenius monoidal functor. 
 
We transfer these results to Hopf monads: we show that under suitable exactness assumptions, a Hopf monad $T$ on a monoidal category has a right adjoint which is also a Hopf comonad, if the object $T\I$ is dualizable as a free $T$-algebra. In particular, if $T\I$ is a Frobenius monoid in the monoidal category of $T$-algebras and $T$ is of descent type, then $T$ is a Frobenius monad and a Frobenius monoidal functor. 

%===================================================================%

\smallskip


\smallskip
{\bf MSC2010}
18D10%Monoidal categories (= multiplicative categories), symmetric monoidal categories, braided categories
;
18C15%Triples (= standard construction, monad or triad), algebras for a triple, homology and derived functors for triples 
; 
18A40%Adjoint functors (universal constructions, reflective subcategories, Kan extensions, etc.) 
; 
16W30%Coalgebras, bialgebras, Hopf algebras; rings, modules, etc. on which these act

\end{abstract}

%========================================================%

%\setcounter{tocdepth}{3}
%\tableofcontents

%========================================================%

\section{Introduction}\label{intro}

Hopf monads on monoidal categories, as introduced and studied in~\cite{blv:hmmc,bv:hm}, are categorical generalizations of Hopf algebras. Several important results on Hopf algebras have their counterpart for Hopf monads, like the fundamental theorem of Hopf modules or Maschke's theorem on semisimplicity. Another important result in the theory of Hopf algebras is that finite-dimensional Hopf algebras are also Frobenius algebras~\cite{larson-sweedler,pareigis}. 

\medskip

It appears thus natural to ask whether Hopf monads are also Frobenius.  Instead of a \emph{finite dimensional} Hopf algebra (hence having a dual, which is itself a Hopf algebra), we shall look for a Hopf monad with \emph{right adjoint}, under the additional proviso that this right adjoint (comonad) is itself \emph{a Hopf comonad}~\cite{cls:hmc}.  We call such a \emph{biHopf monad}. Hopf monads on autonomous categories automatically have right adjoints and are biHopf monads, but Hopf monads on arbitrary monoidal categories are not necessarily biHopf, not even if they have right adjoints, and we provide a simple example of such. 

\medskip

Before proceeding, we need to understand what do we mean by a (bi)Hopf monad $T$ on a monoidal category $\C$ to be Frobenius. This involves the two monoidal structures which play a role in the definition of a Hopf monad.%
\footnote{
Note however that there exists another notion of Hopf monad~\cite{mes-wis:nbhm} on arbitrary categories.
}
 
First, it is a monad, thus a monoid in the \emph{monoidal} category of endofunctors $[\C, \C]$, endowed with composition as tensor product. So one may ask when a Hopf monad is also a Frobenius monad, that is, a Frobenius monoid in $[\C, \C]$~\cite{street:frob}. In particular, it should be self-adjoint. 
 
Second, $T$ carries a comonoidal structure, thus the monoidal structure of its base category $\C$ is also involved. The corresponding notion of interest is that of a Frobenius monoidal functor, which is simultaneously a monoidal and a comonoidal functor, subject to two coherence conditions which ensure that such functor preserves dual objects and Frobenius monoids. 

Hopf monads induced by (tensoring with) finite dimensional Hopf algebras on the category of (finite dimensional) vector spaces are both Frobenius monads and Frobenius monoidal functors. Thus our interest is naturally motivated. 

\medskip

Hopf monads are completely definable in terms of Hopf adjunctions~\cite{blv:hmmc}. BiHopf monads on monoidal categories are completely determined by biHopf adjunctions, that is, triple adjunctions $L\dashv U\dashv R$ with $U$ strong monoidal, such that $L\dashv U$ is a Hopf adjunction and $U\dashv R$ is a coHopf adjunction. The readers should think of $L$ as being the ``induction'' functor and of $R$ as being the ``coinduction'' functor, as these are usually called in Hopf algebra theory. If $U$ is a Frobenius functor (its left and right adjoints are isomorphic), then the associated Hopf monad $T=UL$ is a Frobenius monad, and conversely, for a Frobenius monad $T$, the forgetful functor from the associated category of algebras is Frobenius~\cite{street:frob}. One step further, if $L$ is a Frobenius monoidal functor, then the associated Hopf monad $T=UL$ is again Frobenius monoidal, being a composite of such functors. 

\medskip

During the evolution of this paper, it turned out to be more natural to work with Hopf adjunctions instead of Hopf monads, and subsequently deducing the corresponding results for Hopf monads. 

\medskip

For any triple adjunction $L\dashv U\dashv R$ between monoidal categories, we observe that the pair $(R, L)$ constitutes a linearly distributive functor in the sense of~\cite{cockett-seely:wdc}, and that there is a linearly distributive adjunction $(U,U)\dashv (R,L)$ between degenerate linearly distributive categories, that is, between monoidal categories, if and only if the adjunction $L\dashv U\dashv R$ is biHopf.%
\footnote{%
Our first result connecting biHopf adjunctions and linearly distributive functors should not be surprising. After all, linearly distributive functors capture much of the linear logic structures, while (representations of) Hopf algebras are known to produce models of (noncommutative or/and cyclic) linear logic~\cite{blute:hopf}.%
} 
A linearly distributive functor is a pair formed by a monoidal functor and a comonoidal functor (co)acting on each other, subject to several coherence conditions ensuring that these (co)actions are compatible with the (co)monoidal structure. Each Frobenius monoidal functor between arbitrary monoidal categories produces a degenerate linearly distributive functor, that is, with equal components, where the (co)strengths are provided by the (co)monoidal structure morphisms~\cite{egger:star}. 

\medskip

A triple adjunction $L\dashv U\dashv R$ between symmetric monoidal closed categories, with $U$ strong monoidal and strong closed, has sometimes been called a Wirthm\"uller context~\cite{may-tac}. As $U$ is strong closed if and only if the adjunction $L\dashv U$ is Hopf, the notion of a Wirthm\"uller context can be considered more generally for arbitrary symmetric monoidal categories. One can go even further, by  dropping the symmetry assumption, and considering separately left and right Wirthm\"uller contexts. 

In particular, any Hopf adjunction $L\dashv U$ between monoidal categories with $U$ having right adjoint $R$ determines a Wirthm\"uller context as above. In the quest of an (Wirthm\"uller) isomorphism between the left and the right adjoints of $U$, we start by considering the more general case of having only one (left Hopf) adjunction $L\dashv U$ and looking for the existence of the other adjunction $U\dashv R$, that is, for the right adjoint $R$. We show that under suitable exactness assumptions, the existence of an object $C\in \C$ such that $LC$ is right dual to $L\I$, implies that the functor $L(C \ot -)$ is right adjoint to $U$, and that the resulting adjunction $U\dashv L(C \ot -)$ is left coHopf. 
One should see the object $C$ as the categorified version of the space of integrals for a Hopf algebra, while the duality between $L\I$ and $LC$ is a consequence of the Fundamental Theorem of Hopf modules, applied to the dual of a finite-dimensional Hopf algebra. More details on the categorification of integrals/coinvariants, for Hopf monads on autonomous/monoidal categories, can be found in~\cite{blv:hmmc,bv:hm}. 

The case when $C$ is isomorphic to the unit object $\I$ receives then a special attention, as it exhibits $L\I$ as a Frobenius monoid. And under the exactness assumptions mentioned earlier, this makes $L$ not only both left and right adjoint to $U$, but also a Frobenius monoidal functor. 

\medskip

The corresponding results for Hopf monads are now easy to obtain: first, that any biHopf monad $T$ on a monoidal category, with right adjoint (comonad) denoted by $G$, induces a linearly distributive comonad $(G,T)$; next, that the existence of a dual pair of \emph{free $T$-algebras} $(T\I, TC)$ induces under suitable exactness assumptions a right adjoint Hopf comonad $G=T(C \ot -)$. One step further, $T\I$ being a Frobenius monoid in the monoidal category of $T$-algebras constrains $T$ to be a Frobenius monad, under the same assumptions mentioned above. In particular, $T$ inherits also a monoidal structure, such that this new monoidal structure on $T$ and its old comonoidal structure combine to make $T$ a Frobenius monoidal functor. 

\medskip

The paper is organised as follows: Section~\ref{sect:prelim} briefly reviews the notions of monoidal categories and functors, dual pairs and Frobenius monoids, and recalls the notion of linearly distributive functor from~\cite{cockett-seely:lf}, particularised to the case of monoidal categories. 

The next section shows that biHopf adjunctions produce linearly distributive adjunctions, and linearly distributive adjunctions induce biHopf adjunctions, provided the left adjoint is degenerate (its components coincide). Consequently, biHopf monads produce linearly distributive comonads and vice versa, again in presence of a similar degeneracy condition. 

Section~\ref{sec:main} analyses the {Wirthm\"uller} context associated to a (left) Hopf adjunction $L\dashv U$, yielding an explicit right adjoint for $U$ as a twisted version of $L$. In the particular case when $L\I$ is a Frobenius monoid, $L$ turns to be both left and right adjoint for $U$. As a bonus, the adjunction $U\dashv L$ results to be (left) coHopf, and $L$ a Frobenius monoidal functor. 

The last section converts the results obtained in Sections~\ref{sec:hopf adj lin} and~\ref{sec:main} to Hopf monads. 

A comment on terminology: we shall employ the name ``(co)monoid'' for an object in a monoidal category endowed with an (co)associative (co)multiplication and (co)unit, and reserve the name ``(co)algebra'' for the more general case of an Eilenberg-Moore (co)algebra for a (co)monad. There is however an exception, as the reader might had seen already, in what concerns ``Hopf algebra'', as this is much more familiar than ``Hopf monoid''.

%===================================================================%

\medskip

%===================================================================%

{\bf Acknowledgements}
The author is grateful to the referee, whose helpful insights have significantly improved the previous version of this paper.

%===================================================================%

\section{Preliminaries}\label{sect:prelim}

%========================================================%

\subsection{Monoidal categories and functors}\label{ssect:mon cat funct}

We shall denote by $\ot$ the tensor product of any monoidal category encountered in this paper, while the unit object will be generically denoted by $\I$. We shall in the sequel omit the associativity and unit constraints, writing as the monoidal categories were strict~\cite{maclane}. The identity morphism will be always denoted by $1$, the carrier being obvious from the context. Also, to avoid overcharge in notations, we shall omit labeling natural transformations.

Many of our proofs rely on commutative diagrams. In order to increase their readability, we shall label diagrams by $(N)$ and $(M)$ if these commute by naturality, respectively by monoidal functoriality, as in $(f\ot 1)(1\ot g) = f\ot g = (1\ot g)(f\ot 1)$. Otherwise, we shall refer to previously labeled relations. 

If $\C$ is a monoidal category, the reversed tensor product determines another monoidal structure on $\C$, that we shall denote by $\C^\cop$. The opposite category of $\C$ also becomes monoidal, with either the original monoidal product, in which case we refer to it as $\C^\op$, or with the reversed monoidal product. We shall then use the notation $\C^{\op, \cop}$. All the above mentioned monoidal categories share the same unit object ${\I}$.  

\medskip

The notion of a monoidal functor between monoidal categories is well-known; a classical reference is~\cite{maclane}; the same goes for the notion of comonoidal functor, as well for monoidal and comonoidal natural transformations. When the structure morphisms are isomorphisms, we shall use the term strong monoidal functor. 

%:FMF
We recall from~\cite{day-pastro:Frob} % \label{def:FMF}
 that a functor $F:\C \to \D$ between monoidal categories is called \emph{Frobenius monoidal} if it carries both a monoidal structure 
\[
f_2:FX\ot FY \to F(X \ot Y), \quad f_0: \I \to F\I 
\]
and a comonoidal structure 
\[
F_2: F(X\ot Y) \to FX \ot FY, \quad F_0: F\I\to \I
\]
subject to the compatibility conditions below: 
\begin{align}
\vcenter{
\xymatrix{
FX \ot F(Y \ot Z) \ar[d]_{f_2} \ar[r]^{1\ot F_2} & FX \ot FY \ot FZ \ar[d]^{f_2\ot 1}\\
F(X\ot Y\ot Z) \ar[r]^{F_2} & F(X\ot Y)\ot FZ
}}
\label{eq1:Frob}
\\
\vcenter{
\xymatrix{
F(X \ot Y) \ot FZ \ar[d]_{f_2} \ar[r]^{F_2\ot 1} & FX \ot FY \ot FZ \ar[d]^{1\ot f_2} \\F(X\ot Y\ot Z) \ar[r]^{F_2} & FX\ot F(Y\ot Z) 
}}
\label{eq2:Frob}
\end{align}
The simplest example of a Frobenius monoidal functor is a strong monoidal one, in which case the structural monoidal/comonoidal morphisms are inverses to each other. 

%========================================================%

\subsection{On duality and Frobenius monoids in monoidal categories}\label{ssec:frob mon}

We quickly review below some basics on dual objects in monoidal categories; more details can be found in the references~\cite{kelly:many,kelly-laplaza:coherence,js:braided,freyd-yetter:coherence}. 

\medskip

A \emph{dual pair} in a monoidal category $\C$ consists of two objects $X,Y$, together with a pair of arrows $\e : X \ot Y \to \I$, $\db:{\I}\to Y\ot X$, called evaluation, respectively coevaluation, satisfying the relations 
\[
\vcenter{
\xymatrix@C=35pt@R=15pt{
X
\ar[r]^-{1\ot \db} 
\ar@<-.5ex>@{=}[dr] 
& 
X \ot Y \ot X 
\ar[d]^{\e\ot 1} 
&
Y
\ar[r]^-{\db \ot 1} 
\ar@<-.5ex>@{=}[dr] 
& 
Y \ot X \ot Y 
\ar[d]^{1\ot \e} 
\\ 
& 
X
&
& 
X}}
\]
Then $X$ is called a \emph{left dual} to $Y$, and $Y$ a \emph{right dual} to $X$. 
Each dual pair induces adjunctions $X \ot - \dashv Y \ot -$ and $-\ot Y \dashv -\ot X$. 

\medskip

The notion of a Frobenius monoid makes sense in any monoidal category; references on Frobenius monoids can be found in~\cite{abrams,fs:farmc,lawvere,street:frob}. We recall below only one of the equivalent characterizations, less familiar (dual of~\cite[Theorem~1.6.(b)]{street:frob}), but more suited for our purposes:

%:def:frob algebra
\begin{definition}\label{def:frob}
Let $(\C,\ot,{\I})$ be a monoidal category. A \emph{Frobenius monoid} in $\C$ is an object $A$ such that $(A, d:A \to A \ot A, e:A \to \I)$ is a comonoid, and there are morphisms $u:\I \to A$, $m:A \ot A \to A$ satisfying the Frobenius laws
\begin{equation}\label{eq:Frob_algebra} 
\xymatrix@R=8pt{ A\ot A \ar[rr]^{1\ot d} \ar[dd]_{d\ot 1} \ar[dr]^m && A\ot A \ot A \ar[dd]^{m\ot 1} \\ 
& A \ar[dr]^d & \\
A\ot A \ot A \ar[rr]^{1\ot m} & & A \ot A} 
\end{equation}
and the relations
\[
\xymatrix@C=33pt{
A 
\ar[r]^-{1 \ot u}
\ar@{=}[dr]
&
A \ot A
\ar[d]^{m}
&
A 
\ar[l]_-{u \ot 1}
\ar@{=}[dl]
\\
&
A
&
}
\] 
\end{definition}

Frobenius monoidal functors \emph{preserve} dual pairs and Frobenius monoids~\cite{day-pastro:Frob}; in particular, the unit object $\I$ of the domain category is always mapped to a Frobenius monoid. Notice also that if $A$ is a Frobenius monoid in a \emph{braided} monoidal category $\C$, then both functors $A \ot -,-\ot A:\C\to \C$ obtained by tensoring with $A$ are Frobenius monoidal~\cite{day-pastro:Frob}. 

\medskip

A \emph{Frobenius monad} on an arbitrary category $\C$ is a Frobenius monoid in the monoidal category of endofunctors $[\C,\C]$, with composition as monoidal product and identity functor as unit. 
Each monoid $A$ in a monoidal category $\C$ induces the monads $A\ot -$ and $ -\ot A$ on $\C$. Then $A$ is a Frobenius monoid if and only if $A \ot -$ is a Frobenius monad, equivalently, if and only if $-\ot A$ is a Frobenius monad. 

\medskip

A \emph{Frobenius functor} is a functor $U$ having left adjoint $F$ which is also right adjoint to $U$~\cite{cmz:Frob-book}. One also says that the adjunction is \emph{ambidextrous}. If $U^T$ denotes the forgetful functor from the category of Eilenberg-Moore algebras $\C^T$  for a monad $T$, then $T$ is a Frobenius monad if and only if $U^T$ is a Frobenius functor~\cite{street:frob}.

%========================================================%

\subsection{Linearly distributive functors between monoidal categories} \label{sect:lin funct mon cat}

A pair of functors between monoidal categories, one of them monoidal and the other comonoidal, subject to several coherence conditions, has been called a linearly distributive functor~\cite{cockett-seely:lf},\footnote{Not to be confused with the notion of linear functor, namely an (enriched) functor between categories enriched over the category of $\Bbbk$-vector spaces, for a commutative field $\Bbbk$.} and it makes sense in a more general context than monoidal categories, namely linearly distributive categories.\footnote{These have been introduced by Cockett and Seely in \cite{cockett-seely:wdc:short} and \cite{cockett-seely:wdc} as to provide a categorical settings for linear logic.} Briefly, a linearly distributive category is a category $\C$ equipped with two monoidal structures $(\C,\ot,{\I})$ and $(\C,\pl,\mathbb 0)$, 
and two natural transformations $A\ot (B\pl C)\longrightarrow (A\ot B)\pl C$, $ 
(A\pl B)\ot C \longrightarrow A\pl(B\ot C)$, subject to several naturality coherence conditions that make the two monoidal structures work well together~\cite{cockett-seely:wdc}. Any monoidal category $(\C, \ot, {\I})$ is a (degenerate) linearly distributive category, with $\ot=\pl$ and ${\I}=\mathbb 0$. This is our case of interest, that we shall pursue in the sequel. 

\medskip

A \emph{linearly distributive functor} between monoidal categories $\C$ and $\D$ consists of a pair of functors $R, L : \C\to \D$, such that $(R,r_2 : R X \ot R Y \to R (X\ot Y),r_0 : \I \to R \I)$ is monoidal, $(L,L_2 : L(X \ot  Y) \to L X \ot L Y,  L_0 : L \I \to \I)$ is comonoidal,  
and there are four natural transformations, called strengths and costrengths
\begin{align*} 
&\nu_R^r:R (X\ot Y)\to L X\ot R Y,  \qquad \nu_R^l:R (X\ot Y)\to R X \ot L Y \\
&\nu_L^r:R X \ot L Y \to L (X\ot Y), \qquad \nu_L^l: L X \ot R Y \to L (X\ot Y) 
\end{align*} 
expressing how $R$ and $L$ (co)act on each other, subject to the several coherence conditions \cite{cockett-seely:lf} listed below. These are grouped in such way that a relation of each group is illustrated by a commutative diagram, from which the other relations belonging to the same group can be easily obtained, as follows: the passage $R/L$ corresponds to a move to $\C^\op$, while the passage $r/l$ is obtained for $\C^\cop$. Finally, both changes become simultaneously available in $\C^{\op, \cop}$. 

\begin{alignat}{2}
%:lf1 
\tag{LF1}\label{lf1}&% L is left/right R-module - unit % R is left/right comodule  - counit 
\begin{cases} 
\nu^l_L \circ (1\ot r_0) = 1\\
\nu^r_L \circ (r_0 \ot 1) = 1 \\
(1\ot L_0) \circ \nu^l_R =1 \\
(L_0\ot 1) \circ \nu^r_R =1 \\
\end{cases} 
&
\vcenter{
\xymatrix@C=21pt@R=20pt{
LX 
\ar[d]_{1 \ot r_0} 
\ar@<+.3ex>@{=}[drr] 
&
&
\\
LX \ot R\I 
\ar[rr]_-{\nll} 
& 
&
LX 
&
}
}
\\
%:lf2
\tag{LF2}\label{lf2} & % L is left/right R-module - multiplication % R is left/right comodule  - comultiplication
\begin{cases}
\nu^l_L \circ (\nu^l_L\ot 1) = \nu^l_L \circ (1\ot r_2 )  \\
\nu^r_L \circ (1\ot \nu^r_L) = \nu^r_L \circ (r_2 \ot 1) \\
(\nu^l_R\ot 1) \circ \nu^l_R = (1\ot L_2) \circ \nu^l_R \\
(1\ot \nu^r_R) \circ \nu^r_R = (L_2 \ot 1) \circ \nu^r_R 
\end{cases}
&
\vcenter{
\xymatrix@C=15pt@R=15pt{
LX\ot RY \ot RZ \ar[r]^{\nll\ot 1} \ar[d]_{1 \ot r_2} & L(X \ot Y)\ot RZ \ar[d]^{\nll} \\
LX\ot R(Y \ot Z) \ar[r]_{\nll} & L(X \ot Y \ot Z) 
}
} 
\\
%:lf3
\tag{LF3}\label{lf3} & % R is L-bicomodule % L is R-bimodule - compatibility between left and right (co)actions
\begin{cases}
\nrl \circ (1 \ot \nll) = \nll \circ (\nrl \ot 1) \\
(1\ot \nlr) \circ \nrr = (\nrr \ot 1) \circ \nlr 
\end{cases} 
&
\vcenter{\xymatrix@R=15pt{
RX\ot LY \ot RZ \ar[r]^{1 \ot \nll} \ar[d]_{\nrl \ot 1} & RX \ot L(Y\ot Z) \ar[d]^{\nrl} \\
L(X\ot Y) \ot RZ \ar[r]_{\nll} & LX \ot RY \ot LZ
}
}
\\
%:lf4 % compatibility between actions and coactions: left L-coaction R --> L * R is left R-linear, where the codomain inherits the structure of R-module from L, etc
\tag{LF4}\label{lf4} &
\begin{cases}
(\nll \ot 1) \circ (1 \ot \nlr) = L_2 \circ \nll \\
(1 \ot \nrl) \circ (\nrr \ot 1) = L_2 \circ \nrl \\
(1 \ot \nll) \circ (\nlr\ot 1) = \nlr \circ r_2 \\
(\nrl \ot 1) \circ (1 \ot \nrr) = \nrr \circ r_2 
\end{cases}
&
\vcenter{
\xymatrix@R=15pt{
LX \ot R(Y \ot Z) \ar[r]^{1\ot \nlr} \ar[d]_{\nll} 
& 
LX \ot RY \ot LZ \ar[d]^{\nll \ot 1} 
\\
L(X \ot Y \ot Z) \ar[r]_{L_2} 
& 
L(X\ot Y) \ot LZ
} }
\\
%:lf5
\tag{LF5}\label{lf5}& % multiplication on R is left/right L-colinear % comultiplication on L is left/right R-linear
\begin{cases}
(1 \ot \nll) \circ (L_2 \ot 1) =L_2 \circ \nll  \\
(\nrl \ot 1) \circ (1 \ot L_2) =L_2 \circ \nrl  \\
(r_2 \ot 1) \circ (1\ot \nlr) = \nlr \circ r_2 \\
(1 \ot r_2 ) \circ (\nrr\ot 1) = \nrr \circ r_2 
\end{cases}
&
\vcenter{
\xymatrix@R=15pt{ L(X\ot Y) \ot RZ \ar[r]^{L_2 \ot 1} \ar[d]_{\nll} & LX \ot LY \ot RZ \ar[d]^{1 \ot \nll} \\
L(X \ot Y \ot Z) \ar[r]_{L_2} & LX\ot L(Y \ot Z) 
}
}
\end{alignat}

Linearly distributive functors compose component-wise: given $(R,L):\C  \to \D$ and $(R',L'):\D \to \mathcal E$ linearly distributive functors between monoidal categories, the pair $(R'R, L'L)$ becomes again a linearly distributive functor, the corresponding natural transformations being given by 
\begin{equation*}
\xymatrix{R'R(X\ot Y) \ar[r]^{R'\nrr} & R'(LX \ot RY) \ar[r]^{\nu^r_{R'}} & L'LX \ot R'RY}
\end{equation*}
and similar three more formulas. 

%:linearly distributive natural transformation
Given linearly distributive functors $(R,L), (R',L'):\C \to \D$, a linearly distributive natural transformation $(R,L)\to (R',L')$~\cite{cockett-seely:lf} consists of a monoidal natural transformation $\rho:R \to R'$ and a comonoidal natural transformation $\lambda:L' \to L$ (notice the change of direction!) satisfying the following relations:
\begin{alignat}{2}\tag{LN}\label{ln}
& \begin{cases}
\lambda \circ \nll \circ (1 \ot \rho) = \nll \circ (\lambda \ot 1) \\
\lambda \circ \nrl \circ (\rho \ot 1) = \nrl \circ (1 \ot \lambda) \\
(1 \ot \lambda ) \circ \nlr \circ \rho = (\rho \ot 1) \circ \nlr \\
(\lambda \ot 1) \circ \nrr \circ \rho = (1 \ot \rho) \circ \nrr 
\end{cases}
&
\vcenter{\xymatrix@R=15pt@C=8pt{& L'X \ot RY \ar[dl]_{1 \ot \rho} \ar[dr]^{\lambda \ot 1} & \\
L'X \ot R'Y \ar[d]_{\nll} && LX \ot RY \ar[d]^{\nll} 
 \\
L'(X \ot Y) \ar[rr]^-\lambda && L(X \ot Y) }}
\end{alignat}

There are well-defined notions of vertical and horizontal compositions for linearly distributive natural transformations for which we refer again to \cite{cockett-seely:lf}, such that monoidal categories with linearly distributive functors and linearly distributive natural transformations organise into a 2-category, denoted $\mathbf{MonLinDist}$.

\begin{example}
Any strong monoidal functor $(U,u_2,u_0)$ between monoidal categories induces a linearly distributive functor by $R=L=U$, with (co)strengths given by $\nu_R^l=\nu_R^r=u_2^{-1}$, $\nu_L^l=\nu_L^r=u_2$. More generally, any Frobenius monoidal functor $(F, f_2, f_0, F_2, F_0)$ produces a linearly distributive functor $(R, L)$ with again equal components $R=L=F$, such that $\nrr=\nlr=L_2=F_2$ and $\nlr=\nll=r_2=f_2$ \cite{egger}. 
\end{example}

%========================================================%

In order to motivate the development in Section~\ref{sec:main}, we recall some facts on how a linearly distributive functor $(R,L)$ acts on the unit object $\I$. As the latter is simultaneously a monoid and a comonoid,  it is mapped by the monoidal functor $R$ to the monoid $(R\I, r_2:R\I \ot R\I \to R\I, r_0: \I \to R\I)$ and by the comonoidal functor $L$ to the comonoid $(L\I, L_2: L\I \to L\I \ot L\I, L_0:L\I \to \I)$. 

Half of the relations~\eqref{lf1}-\eqref{lf2} ensure that $\nrr:R\I \to L\I \ot R\I$ defines a left $L\I$-coaction on $R\I$ and $\nlr:R\I \to R\I\ot L\I$ defines a right $L\I$-coaction on $R\I$, while~\eqref{lf3} says that $R\I$ becomes an $L\I$-bicomodule, such that by~\eqref{lf5} the multiplication $r_2:R\I \ot R\I \to R\I$ respects both the left and right $L\I$-coactions.   
Similarly, $L\I$ becomes a $R\I$-comodule using $\nrl$ and $\nll$ as actions, such that the comultiplication $L_2:L\I \to L\I \ot L\I$ is a morphism of $R\I$-bimodules. Then~\eqref{lf4} implies that the left/right $L\I$-coactions on $R\I$ are morphisms of left, respectively right $L\I$-comodules, and similarly for the $R\I$-actions on $L\I$. 

The object $L\I$ is both a left and right dual for $R\I$, with evaluation and coevaluation morphisms 
\vskip-1em
\[
\xymatrix@C=25pt@R=8pt{
\I \ar[r]^{r_0} & R\I \ar[r]^-{\nlr} & R\I \ot L\I, \quad L\I \ot R\I \ar[r]^-{\nll} & L\I \ar[r]^{L_0} & \I
\\
\I \ar[r]^-{r_0} & R\I \ar[r]^-{\nrr} & L\I \ot R\I, \quad R\I \ot L\I \ar[r]^-{\nrl} & L\I \ar[r]^{L_0} & \I
}
\]
Under these dualities, the comultiplication of the comonoid $L\I$ is transposed to the left and right $L\I$-coactions on $R\I$, which transposed again produce the multiplication on $R\I$.

Such a pair $(R\I, L\I)$ formed by a monoid and a comonoid (co)acting on each other was called a cyclic nuclear monoid in~\cite{egger}, or a linear monad in~\cite{cockett-koslowski-seely:lb}.

%========================================================%

\section{Hopf adjunctions and linearly distributive functors}\label{sec:hopf adj lin}

%========================================================%

\subsection{Comonoidal and Hopf adjunctions} \label{ssec: hopf adj}
%:mon adj

Let $\C $ and $\D$ be monoidal categories. In the sequel, we shall denote objects of $\C$ by $X,Y, \ldots$ and the objects of $\D$ by $A, B, \ldots$ to make a clear distinction between them. 

Recall from~\cite{kellydoctrinal} that given an adjunction between $\C$ and $\D$, monoidal structures on the right adjoint are in one-to-one correspondence with comonoidal structures on the left adjoint, such that the unit and the counit of the adjunction become monoidal-comonoidal natural transformations. 

\medskip

We consider throughout that $(U:\D\to \C, u_2:UA \ot UB \to U(A \ot B), u_0:\I \to U\I)$ is a strong monoidal functor, having either a left adjoint $L$ or a right adjoint $R$, depending on the context.

Denote by $\eta^l:1\to UL$ and $\epsilon^l:LU\to 1$ the unit, respectively the counit of the adjunction $L\dashv U$. As mentioned above, $L$ becomes a comonoidal functor, with structure morphisms 
\begin{equation}
\label{eq:induced_colax0} 
\resizebox{.9\textwidth}{!}{%
$\begin{aligned}
& \xymatrix{L_0: L{\I} \ar[r]^-{Lu_0} & LU{\I} \ar[r]^{\epsilon^l} & {\I}}
\\  
& \xymatrix@C=25pt{L_2: L(X\ot Y) \ar[r]^-{L(\eta^l \ot \eta^l)} & L(ULX \ot ULY) \ar[r]^{Lu_{2}} & LU(LX\ot LY) \ar[r]^-{\epsilon^l} & LX \ot LY}
\end{aligned}$
}
\end{equation}
such that the unit and the counit satisfy the diagrams below: 
\begin{equation}
\label{unit-lax-colax}
\vcenter{\xymatrix@R=15pt{
X \ot Y \ar[d]_{\eta^l \ot \eta^l} \ar[r]^-{\eta^l} 
& 
UL(X\ot Y) \ar[d]^{UL_2} 
& 
{\I} \ar[r]^{\eta^l} \ar[dr]_{u_0} 
& 
UL{\I} \ar[d]^{UL_0} 
\\
ULX \ot ULY \ar[r]^{u_2} & U(LX \ot LY) 
&
& U{\I}
\\
L(UA \ot UB) \ar[d]_{Lu_2} \ar[r]^{L_2} & LUA\ot LUB \ar[d]^{\eps^l \ot \eps^l} 
&
L{\I}  \ar[d]_{Lu_0} \ar[dr]^{L_0} & 
\\
LU(A \ot B) \ar[r]^-{\epsilon^l} & A \ot B
&
LU{\I} \ar[r]^{\eps^l}  & {\I} 
}}
\end{equation}

\medskip

For the adjunction $L\dashv U$, the natural composites 
\begin{align} 
%:Hopf operators
\label{left Hopf} 
& \xymatrix@C=20pt{
\hol: L(- \ot U(-)) \ar[r]^-{L_{2}} & L(-) \ot LU(-) \ar[r]^-{1\ot \epsilon^l} & L(-) \ot -
}
\\
\label{right Hopf} 
& \xymatrix@C=20pt{
\hor: L(U(-) \ot -) \ar[r]^-{L_{2}} & LU(-) \ot L(-) \ar[r]^-{\epsilon^l \ot 1} & - \ot L(-)
}
\end{align}
were called  in~\cite{blv:hmmc} the \emph{left Hopf operator}, respectively the \emph{right Hopf operator}. Following \emph{op.~cit.}, $L\dashv U$ is called a \emph{left Hopf adjunction} if $\hol$ is invertible.\footnote{In this case, one also say that \emph{projection formula} holds \cite{shulman}. Recall from~\cite{e-k:cc} that a functor between (left) \emph{closed} monoidal categories is monoidal if and only if it is (left) closed. If additionally it is strong monoidal and right adjoint, then it is strong (left) closed if and only if all (left) Hopf operators are invertible. } Similar terminology goes for the right-handed case; the adjunction $L\dashv U$ is called simply a Hopf adjunction if it is both left and right Hopf adjunction.

The next lemma collect some properties of Hopf operators needed in the sequel; for convenience, these are listed in groups, illustrating an example of each by a commutative diagram, and are easy to prove by standard diagram chasing left to the reader (see also~\cite[Proposition~2.6]{blv:hmmc}).

\begin{lemma} 
For a comonoidal adjunction $L\dashv U$, the Hopf operators satisfy the following relations:  
\begin{alignat}{3}
%:h1l
\label{h1l}&
\begin{cases}
\hol \circ L(1 \ot u_0) = 1 \\
\hor \circ L(u_0 \ot 1) = 1 
\end{cases}
&
\vcenter{\xymatrix@C=65pt@R=15pt{LX \ar@{=}@/_1.8ex/[dr] \ar[r]^-{L(1 \ot u_0)} & L(X \ot U\I) \ar[d]^{\hol} \\
 & LX }} 
&
\\
%:h2l
\label{h2l} & 
\begin{cases}
(\hol \ot 1) \circ \hol = \hol \circ L(1 \ot u_2) \\
(1 \ot \hor ) \circ \hor = \hor \circ L(u_2 \ot 1) 
\end{cases}
& 
\vcenter{\xymatrix@C=22pt@R=15pt{L(X \ot UA \ot UB) \ar[d]_{L(1 \ot u_2)}  \ar[r]^-{\hol} & L(X \ot UA) \ot B \ar[d]^{\hol \ot 1} \\
L(X \ot U(A \ot B)) \ar[r]_-{\hol} & LX \ot A \ot B}} 
&
\\
%:h3l
\label{h3l} &
(\hor \ot 1) \circ \hol = (1 \ot \hol) \circ \hor
& 
\vcenter{\xymatrix@R=15pt{L(UA \ot X \ot UB) \ar[r]^{\hol} \ar[d]_{\hor} & L(UA \ot X) \ot B \ar[d]^{\hor \ot 1} \\
A \ot L(X \ot UB) \ar[r]_{1 \ot \hol} & A \ot LX \ot B}} 
&
\\
%:h4l
\label{h4l} & 
\begin{cases}
\hol \circ L(1\ot \eta^l) =L_2 \\
\hor \circ L(\eta^l\ot 1) = L_2 
\end{cases}
&
\vcenter{\xymatrix@C=46pt@R=15pt{L(X \ot Y) 
\ar@/_1.8ex/[dr]_{L_2} \ar[r]^-{L(1\ot \eta^l )} & L(X \ot ULY) \ar[d]^{\hol} \\
& LX \ot LY }} 
&
\\
%:h5l
\label{h5l} & 
\begin{cases}
U\hol \circ \eta^l = u_2 \circ (\eta^l \ot 1) \\
U \hor \circ \eta^l = u_2 \circ (1 \ot \eta^l) 
\end{cases}
&
\vcenter{\xymatrix@C=47pt@R=15pt{X \ot UA  \ar[r]^-{\eta^l} \ar[d]_{\eta^l \ot 1} & UL(X\ot UA) \ar[d]^{U\hol} \\
ULX \ot UA \ar[r]_{u_2} & U(LX \ot A) }}
&
\\ 
%:h6l
\label{h6l} & \begin{cases} 
(L_2 \ot 1) \circ \hol = (1 \ot \hol ) \circ L_2 \\
(1 \ot L_2) \circ \hor = (\hor \ot 1) \circ L_2 
\end{cases}
&
\vcenter{\xymatrix@R=15pt{L(X \ot Y\ot ULZ ) \ar[r]^{\hol} \ar[d]_{L_2} & L(X \ot Y)\ot LZ  \ar[d]^{L_2 \ot 1} \\
LX \ot L(Y \ot ULZ) \ar[r]_-{1 \ot \hol } & LX \ot LY \ot LZ }}   
&
\\
%:h8l
\label{h8l} & 
\begin{cases}
(L_0\ot 1)\circ \hol = \eps^l \\
(1\ot L_0) \circ \hor = \eps^l 
\end{cases}
&
\vcenter{\xymatrix@C=75pt@R=15pt{LUA \ar[r]^-{\hol} \ar@/_1.8ex/[dr]_{\eps^l} & L\I \ot A  \ar[d]^{L_0 \ot 1} \\
& A}}
&
\end{alignat}
\end{lemma}

Consider also the situation where $U$ has a right adjoint: $(U\dashv R, \eta^r:1\to RU, \eps^r:UR \to 1)$. This time $R$ inherits a monoidal structure $(r_0:\I \to R\I, r_2 :RX \ot RY \to R(X \ot Y))$, given by (the dual of)~\eqref{eq:induced_colax0}. The corresponding diagrams for the monoidality of the unit and the counit will be referred in the sequel as~\textsf{co}\eqref{unit-lax-colax}. 

The left and right coHopf operators associated to the adjunction $U\dashv R$ are
\begin{equation} 
\label{chor}
\vcenter{
\xymatrix@R=10pt{
\chol:  R(-) \ot - \ar[r]^-{1\ot \eta^r} & R(-) \ot RU(-) \ar[r]^{r_2} & R( - \ot U(-)) \\
\chor: - \ot R(-) \ar[r]^-{\eta^r\ot 1 } & RU(-) \ot R(-) \ar[r]^{r_2} & R(U(-) \ot - ) 
}}
\end{equation}

We call the adjunction $U \dashv R$ a \emph{left coHopf adjunction} if the left coHopf operator $\chol$ is invertible, a \emph{right coHopf adjunction} if the right coHopf operator $\chor$ is invertible, and simply a \emph{coHopf adjunction} if it is both left and right coHopf adjunction.\footnote{Recall from~\cite{cls:hmc} that a functor between (left) coclosed monoidal categories is comonoidal if and only if it is (left) coclosed. If additionally it is strong monoidal and left adjoint, then it is strong (left) coclosed if and only if all (left) coHopf operators are invertible.} We do not list anymore the straightforward duals of~\eqref{h1l}-\eqref{h8l}, but instead we make again the convention that whenever we have to refer to them in the sequel, they will appear with the prefix ``\textsf{co}''. 

Finally, say that a {\em biHopf adjunction} is a triple adjunction $L\dashv U \dashv R$ between monoidal categories, with $U$ strong monoidal, $L\dashv U$ a Hopf adjunction and $U \dashv R$ a coHopf adjunction.

%========================================================%

\subsection{BiHopf adjunctions yield linearly distributive functors}\label{ssec:bihopf adj lin}

Linearly distributive functors, as explained earlier, can be seen as a more general version of Frobenius monoidal functors. In our quest for the latter, we show as a first step that biHopf adjunctions \emph{always} produce linearly distributive functors:

\begin{theorem}\label{bihopf is linear}
%:bihopf is linear
Consider a biHopf adjunction $L\dashv U\dashv R:\C \to \D$ between monoidal categories. Then:
\begin{enumerate}
\item The pair $(R, L)$ is a linearly distributive functor. 
\item \label{bihopf}There is an adjunction $(U,U) \dashv (R,L)$ in the 2-category $\mathbf{MonLinDist}$ of monoidal categories, linearly distributive functors and linearly distributive natural transformations, with unit $(\eta^r, \epsilon^l)$ and counit $(\eps^r, \eta^l)$. 
\end{enumerate}
\end{theorem}

\begin{proof}

1. The four strengths and costrengths are defined as follows: 
\begin{equation}\label{(co)strengths for Hopf}
%:(co)strengths for Hopf
\vcenter{
\xymatrix@C=32pt@R=6pt{\nu_R^r:R(X\ot Y) \ar[r]^-{R(\eta^l \ot 1)} & R(ULX \ot Y) \ar[r]^-{\mathfrak H^{r,-1}} & LX \ot RY \\
\nu_R^l:R(X\ot Y) \ar[r]^-{R(1\ot \eta^l)} & R(X \ot ULY) \ar[r]^-{\mathfrak H^{l,-1}} & RX \ot LY \\
\nu_L^r:RX \ot LY \ar[r]^-{\mathbbm H^{r, -1}} & L(URX \ot Y) \ar[r]^-{L(\epsilon^r \ot 1)} & L(X \ot Y)  \\
\nu_L^l:LX \ot RY \ar[r]^-{{\hol}^{-1}} & L(X \ot URY) \ar[r]^-{L(1\ot \epsilon^r )} & L(X \ot Y) } }
\end{equation}
Now we shall check that these form a linearly distributive functor. For each group of relations it is enough to prove only one of them, the remaining  following by the passage to $\C^\op$, $\C^\mathsf{cop}$, or $\C^{\op, \mathsf{cop}}$. 

Proof of~\eqref{lf1}:
\[
\xymatrix@C=40pt@R=15pt{%1
LX 
\ar[dd]^{1 \ot \eta^r}
\ar[dr]^{{\hol}^{-1}}
\ar@{-}@<+.25ex>`r[dddrrr][dddrrr]
\ar@{-}@<-.25ex>`r[dddrrr][dddrrr]
&
&
&
\ar@{}[ddll]|{\eqref{h1l}}
\\%2
&
L(X \ot U\I)
\ar[d]_{L(1 \ot U\eta^r)}
\ar@<+.3ex>@{=}[dr]
\\%3
LX \ot RU\I 
\ar[d]^{1 \ot Ru_0^{-1}}
\ar@{}[r]|{(N)}
&
L(X \ot URU\I)
\ar[d]_{L(1\ot URu_0^{-1})}
\ar[r]^{L(1 \ot \eps^r U)}
&
L(X \ot U\I)
\ar[dr]^{L(1 \ot u_0^{-1})}
\ar@{}[dl]|-{(N)}
\\%4
LX \ot R\I 
\ar[r]_{{\hol}^{-1}}
\ar@{<-}`l[uuu]`[uuu]^{1 \ot r_0}[uuu]
\ar`d[rrr]`[rrr]_{\nll}[rrr]
&
L(X \ot UR\I)
\ar[rr]_{L(1\ot \eps^r)}
&
&
LX
}
\]
Proof of~\eqref{lf2}:
\[
\resizebox{\textwidth}{!}{%
\xymatrix@C=30pt@R=15pt{%1
LX \ot RY \ot RZ
\ar[dd]_{1 \ot r_2}
\ar@{=}
[r]&
LX \ot RY \ot RZ
\ar[r]^{{\hol}^{-1}\ot 1}
\ar[d]_{{\hol}^{-1}}
\ar@{}[dr]|{\eqref{h2l}}
&
L(X \ot URY) \ot RZ
\ar[r]^-{L(1 \ot \eps^r) \ot 1}
\ar[d]^{{\hol}^{-1}}
\ar@{}[dr]|{(N)}
&
L(X \ot Y) \ot RZ
\ar[d]_-{{\hol}^{-1}}
\ar@{<-}`u[ll]`[ll]_{\nll \ot 1}[ll]
\ar`r[dd]`[dd]^{\nll}[dd]
\\%2
&
L(X \ot U(RY \ot RZ))
\ar[r]^{L(1 \ot u_2)}
\ar[d]_{L(1 \ot Ur_2)}
&
L(X \ot URY \ot URZ)
\ar[r]^{L(1 \ot \eps^r\ot 1)}
&
L(X \ot Y \ot URZ)
\ar[d]_{L(1 \ot 1\ot \eps^r)}
\\%3
LX \ot R(Y \ot Z)
\ar[r]^{{\hol}^{-1}}
\ar@{}[uur]|{(N)}
\ar`d[rrr]`[rrr]_{\nll}[rrr]
&
L(X \ot UR(Y \ot Z))
\ar[rr]^{L(1 \ot \eps^r)}
\ar@{}[urr]|{\mathsf{co}\eqref{unit-lax-colax}}
&
&
L(X \ot Y\ot Z)
}}
\]
Proof of~\eqref{lf3}:
\[
\resizebox{\textwidth}{!}{%
\xymatrix@R=20pt@C=35pt{%1
RX \ot LY\ot RZ
\ar[r]^-{1 \ot {\hol}^{-1}}
\ar[d]_{{\hor}^{-1}\ot 1}
\ar@{}[dr]|{\eqref{h3l}}
&
RX \ot L(Y \ot URZ)
\ar[r]^-{1 \ot L(1 \ot \eps^r)}
\ar[d]^{{\hor}^{-1}}
\ar@{}[dr]|{(N)}
&
RX \ot L(Y \ot Z)
\ar[d]^{{\hor}^{-1}}
\ar@{<-}`u[ll]`[ll]_{1 \ot \nll}[ll]
\ar`r[dd]`[dd]^{\nrl}[dd]
\\%2
L(URX \ot Y) \ot RZ
\ar[r]^{{\hol}^{-1}}
\ar[d]_{L(\eps^r \ot 1) \ot 1}
\ar@{}[dr]|{(N)}
&
L(URX \ot Y \ot URZ)
\ar[r]^{L(1 \ot 1 \ot \eps^r)}
\ar[d]^{L(\eps^r \ot 1 \ot 1)}
&
L(URX \ot Y \ot Z)
\ar[d]^{L(\eps^r \ot 1 \ot 1)}
\\%3
L(X \ot Y) \ot RZ
\ar[r]^{{\hol}^{-1}}
\ar@{<-}`l[uu]`[uu]^{\nrl \ot 1}[uu]
&
L(X \ot Y \ot URZ)
\ar[r]^{L(1 \ot 1 \ot \eps^r)}
\ar@{}[ur]|{(M)}
&
L(X \ot Y \ot Z)
\ar@{<-}`d[ll]`[ll]^{\nll}[ll]
}}
\]
Proof of~\eqref{lf4}:
\begin{equation}\label{eq:proof of lf4}
\resizebox{.92\textwidth}{!}{$\vcenter{\xymatrix@C=37pt@R=15pt{%1
LX \ot R(Y \ot Z) 
\ar[r]^{{\hol}^{-1}}
\ar[d]_{1 \ot R(1 \ot \eta^l)}
\ar@<+.1ex>`l[dd]`[dd]_{1 \ot \nlr} [dd]
\ar@{}[dr]|{(N)}
&
L(X \ot UR(Y \ot Z))
\ar[d]^{L(1 \ot UR(1 \ot \eta^l))}
\ar[r]^-{L(1 \ot \eps^r)}
\ar@{}[dr]^(.6){(N)}
&
L(X \ot Y \ot Z)
\ar[d]^{L(1 \ot 1\ot \eta^l)}
\ar@{<-}`u[ll]`[ll]_{\nll}[ll]
\\%2
LX \ot R(Y \ot ULZ)
\ar[r]^{{\hol}^{-1}}
\ar[d]_{1 \ot {\chol}^{-1}}
\ar@{}[dr]|{(N)}
&
L(Z \ot UR(Y \ot ULX))
\ar[r]^-{L(1 \ot \eps^r)}
\ar[d]^{L(1 \ot U{\chol}^{-1})}
&
L(X \ot Y \ot ULZ)
\ar@{=}[dd]
\\%3
LX \ot RY \ot LZ
\ar[r]^{{\hol}^{-1}}
\ar[d]_{{\hol}^{-1} \ot 1}
\ar@<+.4ex>`l[dddrr]`[dddrr]_{\nll \ot 1}`[ddrr][ddrr]
\ar@{}[dr]|{\eqref{h2l}}
&
L(X \ot U(RY \ot LZ))
\ar[d]^{L(1 \ot u_2^{-1})}
\ar@{}[r]|(.65){\mathsf{co}\eqref{h5l}}
\ar@{}[r]|(1.2){\eqref{h4l}}
&
\\%4
L(X \ot URY) \ot LZ 
\ar[r]^{{\hol}^{-1}}
\ar@{-}@<-.25ex>`d[dr][dr]
\ar@{-}@<+.25ex>`d[dr][dr]
&
L(X \ot URY \ot ULZ)
\ar[r]^-{L(1 \ot \eps^r \ot 1)}
\ar[d]^{\hol}
\ar@{}[dr]^(.6){(N)}
&
L(X \ot Y \ot ULZ)
\ar[d]^{\hol}
\\%5
&
L(X \ot URY) \ot LZ 
\ar[r]^-{L(1 \ot \eps^r) \ot 1}
&
L(X \ot Y) \ot LZ
\ar@{<-}`r[uuuu]`[uuuu]_{L_2}[uuuu]
\\
&
&
}}$}
\end{equation}
Proof of~\eqref{lf5}:
\begin{equation}\label{eq:proof of lf5}
\vcenter{\xymatrix@R=20pt@C=55pt{%1
L(X \ot Y) \ot RZ
\ar[r]^{L_2 \ot 1}
\ar[d]_{{\hol}^{-1}}
&
LX \ot LY \ot RZ 
\ar[d]^{1 \ot {\hol}^{-1}}
\\%2
L(X \ot Y \ot URZ)
\ar[r]^{L_2}
\ar[d]|{L(1 \ot 1 \ot \eps^r)}
\ar@{}[dr]|{(N)}
\ar@{}[ur]|{\eqref{h6l}}
&
LX \ot L(Y \ot URZ)
\ar[d]|{1 \ot L(1 \ot \eps^r)}
\\%3
L(X \ot Y \ot Z) 
\ar[r]^{L_2}
\ar@{<-}`l[uu]`[uu]^{\nll}[uu]
&
LX \ot L(Y \ot Z) 
\ar`r[uu]`[uu]_{1 \ot \nll}[uu]
}}
\end{equation}

2. The only thing to verify is that the unit $(\eta^r, \eps^l)$ and the counit $(\eps^r, \eta^l)$ are indeed linearly distributive natural transformations. That is, they should satisfy the following relations obtained from~\eqref{ln}: 
\begin{align}
%:lin-unit
& \label{lin-unit}
(\eta^r ,\eps^l) \text{ linear: } \begin{cases}
(\eps^l \ot 1) \circ \nrr \circ Ru^{-1}_2 \circ \eta^r = 1\ot \eta^r \\
(1\ot \eps^l ) \circ \nlr \circ Ru^{-1}_2 \circ \eta^r = \eta^r \ot 1 \\
\eps^l \circ Lu_2 \circ \nrl \circ (\eta^r \ot 1 ) = 1 \ot \eps^l \\
\eps^l \circ Lu_2 \circ \nll \circ (1 \ot \eta^r) = \eps^l \ot 1
\end{cases} 
\\
%:lin-counit
& \label{lin-counit}
(\eps^r, \eta^l) \text{ linear: } \begin{cases}
(\eta^l \ot 1) \circ \eps^r = (1 \ot \eps^r)\circ u_2^{-1} \circ U\nrr \\
(1 \ot \eta^l ) \circ \eps^r = (\eps^r\ot 1 )\circ u_2^{-1} \circ U\nlr \\
U \nrl \circ u_2 \circ (1 \ot \eta^l) = \eta^l \circ (\eps^r \ot 1) \\
U \nll \circ u_2 \circ (\eta^l \ot 1) = \eta^l \circ (1 \ot \eps^r ) 
\end{cases}
\end{align}
It is enough to check only one of each couple of four relations, the rest following by passage either to $\C^\op$, $\C^\cop$ or to $\C^{\op,\cop}$. For example, the first relation in~\eqref{lin-unit} is proved in the diagram below:
\[
\xymatrix@C=40pt@R=15pt{%1
A\ot B 
\ar@{}[drr]|{\mathsf{co}\eqref{unit-lax-colax}}
\ar[d]_{1\ot \eta^r}  
\ar[rr]^-{\eta^r} 
& 
& 
RU(A\ot B) 
\ar[d]^{Ru_2^{-1}}  
\\%2
A \ot RUB 
\ar@/_7ex/[ddrr]^{\chor} 
\ar[r]^-{\eta^r \ot 1} 
&
RUA \ot RUB 
\ar@{}[dd]|(.4){\eqref{chor}} 
\ar[r]^{r_2} 
& 
R(UA \ot UB)  
\ar@<+.8ex>@{-}`l[dd]`[dd][dd] 
\ar@<+.3ex>@{-}`l[dd]`[dd][dd]  
\ar[d]^{R(\eta^l U \ot 1)}  
\\%3
&  
& 
R(ULUA\ot UB) 
\ar[dr]^{{\chor}^{-1} } 
\ar[d]^{R(U\eps^l \ot 1)} 
&
\\%4
& 
&
R(UA \ot UB) 
\ar@{}[r]|{(N)} 
\ar[dr]^{{\chor}^{-1}} 
& 
LUA \ot RUB 
\ar@{<-}`u[uul]_{\nrr}[uul] 
\ar[d]^{\epsilon^l \ot 1} 
\\%5
&
&
&
A \ot RUB
\ar@<-1.25pt>@{-}`l[uuulll][uuulll] 
\ar@<+1.25pt>@{-}`l[uuulll][uuulll] 
}
\]
while the proof of the first relation in~\eqref{lin-counit} is given below:
\[
\xymatrix@R=7pt{
UR(X \ot Y) 
\ar[ddd]_{UR(\eta^l \ot 1)}
\ar[rrr]^{\eps^r}
\ar@{}[ddrrr]|{(N)}
&
&
&
X \ot Y
\ar[ddd]^{\eta^l \ot 1}
\\
&
&
&
\\
&
&
&
\\
UR(ULX \ot Y) 
\ar[r]_{U{\chor}^{-1}}
\ar@<-.8ex>`u[urrr]`[rrr]^{\eps^r}[rrr]
&
U(LX \ot RY)
\ar[r]_{u_2^{-1}}
\ar@{}[ur]|{\mathsf{co}\eqref{h5l}}
&
ULX \ot URY 
\ar[r]_{1 \ot \eps^r}
&
ULX \ot Y
\\
}
\]
\end{proof}

\begin{remark}
Consider a linearly distributive adjunction $(U,V)\dashv (R,L)$ between monoidal categories, with unit $(\eta^r:1\to RU, \eps^l:LV \to 1)$ and counit $(\eps^r:UR \to 1, \eta^l : 1 \to VL)$. In particular, $(U\dashv R,\eta^r,\eps^r)$ is a monoidal adjunction and $(L\dashv V,\eta^l, \eps^l)$ a comonoidal adjunction; consequently, both $U$ and $V$ are strong monoidal functors.  

\noindent Assume that $U=V$; then the linearly distributive adjunction $(U,U)\dashv (R,L)$ induces a biHopf triple adjunction $L\dashv U\dashv R$. More precisely, the formulas below provide inverses for the Hopf and coHopf operators:
\[
\xymatrix@C=40pt@R=4pt{
{\hol}^{-1}: LX \ot A 
\ar[r]^-{1 \ot \eta^r} 
& 
LX \ot RUA 
\ar[r]^{\nu_L^l}
&
L(X \ot UA)
\\
\mathbbm H^{r, -1}: A \ot LX 
\ar[r]^-{\eta^r \ot 1} 
& 
RUA \ot LX
\ar[r]^{\nrl}
&
L(UA \ot X)
\\
\quad \mathfrak H^{l,-1}: R(X \ot UA)
\ar[r]^-{\nlr}
&
RX \ot LUA 
\ar[r]^{1 \ot \eps^l}
&
RX \ot A
\\
\quad \mathfrak H^{r,-1}: R(UA \ot X)
\ar[r]^-{\nrr}
&
LUA \ot RX 
\ar[r]^{\eps^l \ot 1}
&
A \ot RX
}
\]
We provide below the computations for $\hol$ and ${\hol}^{-1}$:

\[
\resizebox{\textwidth}{!}{
\xymatrix@R=25pt{
LX \ot A 
\ar[r]^-{1 \ot \eta^r}
\ar@/_2.85ex/[dddr]|{1 \ot \eta^r \ot 1}
\ar[ddd]|-{1 \ot r_0 \ot 1}
\ar@{}[dr]|{\eqref{lin-unit}}
&
LX \ot RUA
\ar[r]^\nll
\ar[d]|{1 \ot R(u_0\ot 1)} %recall that $u_0\ot 1$ and $u_{2,\I,-}$ are inverses
&
L(X \ot UA)
\ar[d]|{L(1 \ot u_0)\ot 1}
\ar@{}[dl]|{(N)}
\ar@<.25ex>@{-}`r[dr][dr]
\ar@<-.25ex>@{-}`r[dr][dr]
\ar@{<-}`u[ll]`[ll]_{{\hol}^{-1}}[ll]
\\
&
LX \ot R(U\I \ot UA) 
\ar[d]^{1 \ot \nlr}
\ar[r]^{\nll}
\ar@{}[dr]|{\eqref{lf4}}
&
L(X \ot U\I \ot UA) 
\ar[d]^{L_2}
\ar[r]^-{L(1 \ot u_0^{-1} \ot 1)}
&
L(X \ot UA)
\ar[d]^{L_2}
\ar`r[ddd]`[ddd]_{\hol}[ddd]
\\
&
LX \ot RU\I \ot LUA
\ar[d]^{1 \ot 1\ot \eps^l}
\ar[r]^{\nll \ot 1}
\ar@{}[l]|(1.15){\eqref{lf1}}
\ar@{}[dr]|{(M)}
&
L(X\ot U\I) \ot LUA \quad 
\ar[d]^{1 \ot \eps^l}
&
LX \ot LUA
\ar[dd]^{1 \ot \eps^l}
\\
LX \ot R\I \ot A
\ar[r]^-{1 \ot Ru_0 \ot 1}
\ar@{}[ur]|{\mathsf{co}\eqref{unit-lax-colax}}
\ar[d]|{\nll \ot 1}
&
LX \ot RU\I \ot A
\ar[r]^{\nll \ot 1}
\ar@{}[dr]|{(N)}
&
L(X \ot U\I) \ot A
\ar@/_2ex/[dr]^-{\ L(1\ot u_0^{-1})\ot 1}
\ar@{}[uur]|{(N)}
&
\\
LX \ot A
\ar@<-.25ex>@{-}`l[uuuu]`[uuuu][uuuu]
\ar@<+.25ex>@{-}`l[uuuu]`[uuuu][uuuu]
\ar@<-.25ex>@{=}[rrr]
&
&
&
LX \ot A
&
}}
\]
\[
\resizebox{\textwidth}{!}{
\xymatrix@R=20pt{ %
L(X \ot UA) 
\ar[d]|-{L(\eta^l \ot 1)}
\ar[rr]^{L(1 \ot U\eta^r)}
\ar@<+.25ex>`l[ddddd]`[dddddr]`[ddddr]_{\hol}[ddddr]
\ar@<-.25ex>@{-}`u[rrr]`[rrr][rrr]
\ar@<+.25ex>@{-}`u[rrr]`[rrr][rrr]
&
\ar@{}[ddr]|{(M)}
&
L(X \ot URUA)
\ar[dd]^{L(\eta^l \ot 1)}
\ar[r]^{L(1 \ot \eps^r U)}
\ar@{}[dddr]|{\eqref{lin-counit}}
& 
L(X \ot UA) 
\ar[ddd]^{L\eta^l}
\ar@<-.25ex>@{-}`r[dddd]`[dddd][dddd]
\ar@<+.25ex>@{-}`r[dddd]`[dddd][dddd]
\\
L(ULX \ot UA)
\ar[d]|-{L(1 \ot \eta^l U)}
\ar@<-.25ex>@{-}`r[dr][dr]
\ar@<+.25ex>@{-}`r[dr][dr]
&
&
\\
L(ULX \ot ULUA)
\ar[d]_{Lu_2}
\ar[r]^{L(1 \ot U\eps^l)}
\ar@{}[dr]|{(N)}
& 
L(ULX \ot UA)
\ar[d]^{Lu_2}
\ar[r]^{L(1 \ot U\eta^r)}
\ar@{}[dr]|{(N)}
&
L(ULX \ot URUA)
\ar[d]^{Lu_2}
\\
LU(LX \ot LUA)
\ar[d]_-{\eps^l}
\ar[r]^{LU(1 \ot \eps^l)}
\ar@{}[dr]|{(N)}
&
LU(LX \ot A)
\ar[d]^{\eps^l}
\ar[r]^{LU(1\ot \eta^r)}
\ar@{}[dr]|{(N)}
&
LU(LX \ot RUA)
\ar[d]^{\eps^l}
\ar[r]^{LU\nll}
\ar@{}[dr]|{(N)}
& 
LUL(X \ot UA)
\ar[d]^{\eps^l L}
\\
LX \ot LUA
\ar[r]^-{1 \ot \eps^l}
\ar@{}[dr]|{\eqref{eq:induced_colax0},\eqref{left Hopf}}
&
LX \ot A
\ar[r]^-{1 \ot \eta^r} 
\ar@<+.5ex>`d[drr]`[urr]_{{\mathbbm H}^{l,-1}}[rr]
&
LX \ot RUA
\ar[r]^-{\nll}
&
L(X \ot UA)
\\
&
&
&
&
}}
\]
The invertibility of the other (co)Hopf operators can be similarly checked. 
\end{remark}

\noindent In view of the previous remark, it seems natural to ask the following question: 
For a linearly distributive adjunction $(U,V)\dashv (R,L)$ between monoidal categories, is it always the case that not only $U$ and $V$ are strong monoidal, but also that $(U,V)$ is a degenerate linearly distributive functor, in the sense that $U$ and $V$ are isomorphic as (strong) monoidal functors (in the spirit of~\cite{kellydoctrinal})?
Notice that~\cite[Proposition~15]{balan:frob} gives a positive answer, under the additional assumption that the categories involved are autonomous.

%========================================================%

\section{Hopf adjunctions and Frobenius-type properties}\label{sec:main}

In this section we consider a strong monoidal functor $U$ between monoidal categories having left adjoint $L$, such that the resulting adjunction $L\dashv U$ is (left) Hopf.  

We shall provide conditions for $U$ to have a right adjoint $R$. Such a triple adjunction $L\dashv U\dashv R$ is an example of a (left) \emph{Wirthm\"uller context}~\cite{may-tac}.%
\footnote{
More precisely, a Wirthm\"uller context has only been considered in the situation where the categories involved were (symmetric) monoidal closed. But the closed structure appears only in the requirement that $U$ to be strong closed, which as said earlier, can be substituted by the more convenient formulation of Hopf adjunction. 
} % 
Additionally, we shall see that $U\dashv R$ is a (left) coHopf adjunction. 

Next, we shall look on when the left and right adjoints of are isomorphic; that is, when $U$ is a Frobenius functor. As a bonus, we shall obtain that $L$ is also a Frobenius monoidal functor. 

%========================================================%

\subsection{The Wirthm\"uller isomorphism.}\label{ssec:wirth}

Let $L \dashv U$ be a left Hopf adjunction. %The comonoidal structure of $L$ induces on  $L\I$ a comonoid structure, with comultiplication $L_2:L\I \to L\I \ot L\I$ and counit $L_0:L\I \to \I$, as observed earlier, such that for any object $X$ of $\C$, $LX$ becomes (both left and) right $L\I$-comodule via $L_2$. 
Motivated by the theory of (finite dimensional) Hopf algebras, but also by the development in~\cite{may-tac}, we shall assume the existence of an object $C$ of $\C$, endowed with a pair of morphisms $u:\I \to LC$, $m: L\I \ot LC \to L\I$, subject to the following relations, labeled for later use: 
\begin{equation}\label{LC}
\vcenter{
\xymatrix@C=11.515pt@R=15pt{%1
L\I \ot LC 
\ar[rr]^-{L_2 \ot 1} 
\ar[dddd]_-{1 \ot L_2} 
\ar[ddr]^{ m}
&
&
L\I \ot L\I \ot LC 
\ar[dddd]^{1 \ot  m}
&
L\I 
\ar[rr]^-{1 \ot u}
\ar@<-0.25ex>@{-}`d[drr][drr]
\ar@<0.25ex>@{-}`d[drr][drr]
\ar@{}[drr]|{\mathrm{(iii)}}
&
&
L\I \ot LC 
\ar[d]^{ m}
\\%2
&
&
&
&
&
L\I
\\%3
&
L\I
\ar[ddr]^{L_2}
\ar@{}[uur]|-{\mathrm{(i)}}
\ar@{}[ddl]|-{\mathrm{(ii)}}
&
&
LC
\ar[r]^-{u \ot 1}
\ar@<-0.25ex>@{-}`d[ddr][ddr]
\ar@<0.25ex>@{-}`d[ddr][ddr]
\ar@{}[ddrr]|{\mathrm{(iv)}}
& 
LC \ot LC 
\ar[r]^-{L_2 \ot 1}
& 
LC \ot L\I \ot LC
\ar[dd]^{1 \ot  m}
\\%4
&&&&&
\\%5
L\I \ot LC \ot L\I
\ar[rr]^{ m \ot 1}
&
&
L\I \ot L\I
&
&
LC 
\ar[r]^{L_2}
&
LC \ot L\I
}}
\end{equation}
Then $LC$ becomes a right dual for $L\I$,%
\footnote{%
Actually, one can easily check that there is a one-to-one correspondence between morphisms $u, m$ satisfying~\eqref{LC}, and the structure of a right dual for $L\I$ on $LC$, such that $L_2:LC \to LC \ot L\I$ is the two-way transpose of $L_2:L\I \to L\I \ot L\I$. Notice also that $u:\I \to LC$ is the transpose of $L_0:L\I \to \I$ under the above duality. 
} %
with evaluation and coevaluation morphisms given by \newline
\begin{equation}\label{LC-right-dual-LI}
\begin{aligned}
&
\db:%
\xymatrix@C=35pt{
\I \ar[r]^u 
& 
LC \ar[r]^-{L_2} 
& 
LC \ot L\I 
}
\\[-2\jot]
&
\e:%
\xymatrix@C=35pt{
L\I \ot LC 
\ar[r]^-{ m}
&
L\I 
\ar[r]^{L_0}
& 
\I  
}
\end{aligned}
\end{equation}
The next diagrams show that the usual triangle equations are satisfied.
\[
\resizebox{\textwidth}{!}{%
\xymatrix@R=25pt{
L\I 
\ar[r]^-{1 \ot u} 
\ar@<-0.25ex>@{-}`d[dr][dr]
\ar@<0.25ex>@{-}`d[dr][dr]
& 
L\I \ot LC 
\ar[r]^-{1 \ot L_2} 
\ar[d]^{ m}
\ar@{}[dl]|{\hyperref[LC]{\eqref{LC}\text{-(iii)}}}
& 
L\I \ot LC \ot L\I
\ar[d]^{ m \ot 1} 
\ar@{}[dl]|{\hyperref[LC]{\eqref{LC}\text{-(ii)}}}
\ar@{<-}`u[ll]`[ll]_{1 \ot \db}[ll]
\ar`r[dd]`[dd]^{\e \ot 1}[dd]
\\
&
L\I 
\ar[r]^-{L_2}
\ar@<-0.25ex>@{-}`d[dr][dr]
\ar@<0.25ex>@{-}`d[dr][dr]
&
L\I \ot L\I
\ar[d]^{L_0 \ot 1}
\\
&
&
L\I 
} 
\xymatrix@R=15pt{
LC
\ar[r]^-{u \ot 1} 
\ar@<-0.25ex>@{-}`d[dr][dr]
\ar@<0.25ex>@{-}`d[dr][dr]
\ar@{}[drr]|{\hyperref[LC]{\eqref{LC}\text{-(iv)}}}
& 
LC \ot L\I 
\ar[r]^-{L_2 \ot 1} 
& 
LC \ot L\I \ot LC
\ar[d]^{1 \ot  m} 
\ar@{<-}`u[ll]`[ll]_{\db\ot 1}[ll]
\ar`r[dd]`[dd]^{1 \ot \e}[dd]
\\
& 
LC
\ar[r]^-{L_2}
\ar@<-0.25ex>@{-}`d[dr][dr]
\ar@<0.25ex>@{-}`d[dr][dr]
&
LC \ot L\I
\ar[d]^{1 \ot L_0}
\\
&
&
LC 
}}
\]

\begin{theorem}\label{frob1}

Let $ L\dashv U:\D \to \C $ be a left Hopf adjunction between monoidal categories. Assume the following:
\begin{enumerate}
\item There is an object $C$ of $\C$ and morphisms $u:\I \to LC$, $ m:L\I \ot LC \to L\I$ satisfying~\eqref{LC}.
\item $L$ is precomonadic.
\item For all $X$, $L(1 \ot \eta^l ): L(C\ot X) \to L(C \ot ULX)$ is a monomorphism.%
\footnote{Using~\eqref{h4l}, this is equivalent with $L_2:L(C \ot X) \to LC \ot LX$ being a monomorphism.}\label{cond2}
\end{enumerate} 
Then the functor $L(C \ot -)$ is right adjoint to $U$. 

\end{theorem}

\begin{proof}

Using that $L\dashv U$ is left Hopf, observe that the dual pair $(L\I, LC)$ induces the adjunction 
\begin{equation}\label{eq:adj}
\xymatrix@C=15pt{
LU
\ar[r]^-{\hol}_-{\cong}
& 
L\I \ot - 
\ar@{}[r]|{\mathlarger{\mathlarger{\mathlarger{\dashv}}}}
& 
LC \ot - \
\ar[r]^{{\hol}^{-1}}_-{\cong}
&
\ L(C \ot U(-))
}
\end{equation}
We have therefore the situation
\[
\xymatrix@C=60pt{
\mathcal D 
\ar[r]^U
\ar@<.6ex>[dr]^{LU}
\ar@{}[dr]|{\small{\rotatedadj}}
& 
\mathcal C 
\ar@<-.7ex>[d]_-{L}
\ar@{}[d]|{\tiny{\dashv}}
\\
& 
\mathcal D 
\ar@<.8ex>[ul]^{L(C \ot U(-))}
\ar@<-.75ex>[u]_{U}
}
\]
for which we shall appeal to the following (dual) version of Dubuc's Adjoint \linebreak Triangle Theorem~\cite{dubuc}:

\begin{theorem}\label{prop:dubuc-adj-tri-thm}
Consider two adjunctions $(L \dashv U:\mathcal D \to \mathcal C, \eta:1 \to UL, \eps:LU \to 1)$, $(L'\dashv U':\mathcal C \to \mathcal B, \eta':1 \to U'L', \eps':L'U' \to 1)$, and a functor $H:\mathcal B \to \mathcal C$ such that $LH \cong L'$. Denote by 
\vskip-1em
$$
\theta: \xymatrix@1@C=25pt{HU' \ar[r]^-{\eta HU'} & ULHU' \cong UL'U' \ar[r]^-{U\eps'} & U}
$$
the mate of $LH \cong L'$ under these adjunctions. If $L$ is precomonadic and the equaliser of the pair 
\vskip-1em
\begin{equation}\label{dubuc}
\xymatrix@C=35pt@R=10pt{
U'LX
%\ar[dr]_-{\eta'U'L} 
\ar[rr]^{U'L \eta}
& 
& 
U'LULX
\\
&
U'L'U'LX
{\cong}
U'LHU'LX
\ar@{<-}[]!<0ex,-4ex>;[ul]!<-6ex,1ex>^-{\eta'U'L}  
\ar[]!<3ex,-3.5ex>;[ur]!<5ex,1ex>_-{U'L\theta L} 
&
}
\end{equation}
exists for every $X\in \mathcal B$, then $H$ has a right adjoint, given as the equaliser of~\eqref{dubuc}.
\end{theorem}

Consequently, the right adjoint of $U$, if exists, should be the equaliser of $L(1 \ot UL\eta^l)$ and of the composite $ L(1 \ot UL\theta L) \circ \eta'$, where $\eta'$ denotes the unit of the adjunction~\eqref{eq:adj}, and $\theta: UL(C \ot U(-)) \to U$ is the mate of the identity arrow on $LU$. 
The computations below show that the composite in question is in fact $L(1 \ot \eta^l UL)$. 
\[
\resizebox{1\textwidth}{!}{
\xymatrix@R=35pt@C=15pt{
L(C \ot ULX)
\ar[rr]^-{u \ot 1}
\ar[d]^{\hol}
&&
LC \ot L(C \ot ULX)
\ar[r]^-{L_2 \ot 1}
\ar[dr]^{{\hol}^{-1}}
\ar[d]^{L(1 \ot \eta^l) \ot 1}
&
LC \ot L\I \ot L(C \ot ULX)
\ar[r]^-{1 \ot {\hol}^{-1}}
&
LC \ot LUL(C \ot ULX)
\ar[r]^-{{\hol}^{-1}}
\ar@{}[ddr]|(.15){\eqref{h4l}}
& 
L(C \ot ULUL(C \ot ULX))
\ar@{<-}`u[lllll]`[lllll]_{\eta'}[lllll]
\ar[dd]^{L(1 \ot UL\eta^l UL)}
\ar`r[dddddr]`[ddddd]_(.3){L(1 \ot UL\theta L)}[ddddd]
&
\\
%2
LC \ot LX 
\ar[dr]^{u \ot 1 \ot 1}
\ar@{-}@<+.25ex>`d[ddddrr][ddddrr]
\ar@{-}@<-.25ex>`d[ddddrr][ddddrr]
\ar@{}[ddddr]|{\hyperref[LC]{\eqref{LC}\text{-(iv)}}}
&
&
L(C \ot UL\I) \ot L(C \ot ULX)
\ar[d]^{1 \ot \hol}
&
L(C \ot UL(C \ot ULX) 
\ar[ur]^{L_2}
\ar@{}[u]|{\eqref{h6l}}
\ar[urr]_{L(1 \ot \eta^l UL)}
\ar[dr]^{L(1 \ot \eta^l UL)}
\ar[d]^{L(1 \ot \eta^l \ot 1)}
&
&
&
\\
%3
&
LC \ot LC \ot LX 
\ar[r]_{L(1 \ot \eta^l) \ot 1 \ot 1}
\ar@{}[ur]|{(M)}
\ar[ddr]^{L_2 \ot 1 \ot 1}
\ar@{}[dr]|{\eqref{h4l}}
&
L(C \ot UL\I) \ot LC \ot LX
\ar[dr]^-{{\hol}^{-1}}
\ar@{}[ur]|{(N)}%
\ar[dd]^{{\hol}^{-1} \ot 1\ot 1}
&
L(C \ot UL\I \ot UL(C \ot ULX))
\ar[dr]^{L(1 \ot u_2)}
\ar@{}[r]|{\eqref{h5l}}
\ar[d]^{L(1 \ot 1 \ot U\hol)}
&
L(C \ot ULUL(C \ot ULX))
\ar@{}[uur]|{(N)}%
\ar@{}[ddr]|{(N)}%
\ar[d]^{L(1 \ot U\hol)}
\ar[r]^{L(1 \ot \eta^l)}
&
L(C \ot ULULUL(C \ot ULX))
\ar[d]^{L(1 \ot ULU\hol)}
&
\\
%4
&
&
&
L(C \ot UL\I \ot U(LC \ot LX))
\ar[dr]^{L(1 \ot u_2)}
\ar@{}[r]|{(N)}
&
L(C \ot U(L\I \ot L(C \ot ULX))
\ar[d]^{L(1 \ot U(1 \ot \hol))}
&
L(C \ot ULU(L\I \ot L(C \ot ULX)))
\ar[d]^{L(1 \ot ULU(1 \ot \hol))}
&
\\
%5
&
&
LC \ot L\I \ot LC \ot LX
\ar[rr]^{{\hol}^{-1}}
\ar@{}[ur]|{\eqref{h2l}}
\ar[d]^{1 \ot \e \ot 1}
&
&
L(C \ot U(L\I \ot LC \ot LX))
\ar[d]^{L(1 \ot U(\e \ot 1))}
&
L(C \ot ULU(L\I \ot LC \ot LX))
\ar[d]^{L(1 \ot ULU(\e\ot 1))}
&
\\
&
&
LC \ot LX
\ar[rr]^{{\hol}^{-1}}
\ar@{}[urr]|{(N)}
&
&
L(C \ot ULX)
\ar[r]^{L(1 \ot \eta^l UL)}
&
L(C \ot ULULX)
&
}
}
\]
Therefore, if the right adjoint of $U$ exists, it must be the equaliser of 
\[
\xymatrix@C=60pt{
L(C \ot UL(-)) 
\ar@<+.5ex>[r]^-{L(1 \ot UL\eta^l)}
\ar@<-.5ex>[r]_-{L(1 \ot \eta^l UL)}
& 
L(C \ot ULUL(-))
}
\]
The functor $L$ being precomonadic means that for each $X$, the diagram
\begin{equation}\label{corefl eq}
\xymatrix@C=45pt{
X 
\ar[r]^{\eta^l}
&
ULX 
\ar@<+.5ex>[r]^-{UL\eta^l}
\ar@<-.5ex>[r]_-{\eta^l UL}
& 
ULULX
}
\end{equation}
is a (coreflexive) equaliser.%
\footnote{%
A precomonadic functor is also said to be {\em of descent type}. Equivalent conditions for a left adjoint functor to be precomonadic are: (i) the unit $\eta^l$ is a regular monomorphism; (ii) the comparison functor into the Eilenberg-Moore category of the comonad induced by the adjunction, is fully faithful (see for example~\cite{ttt}).  
} %
Henceforth the proof is finished if we show that $L(C \ot -)$ preserves it. 

\medskip\noindent 
Consider thus an arrow $f: A \to L(C \ot ULX)$ such that 
\begin{equation}\label{f-eq}
L(1 \ot UL\eta^l ) \circ f = L(1 \ot \eta^l UL)\circ f
\end{equation}
holds, and denote by $g: UA \to ULX$ the two-way transpose of $f$ under the adjunctions $LU\dashv L(C \ot U(-))$ and $L \dashv  U$. That is, $g$ is the composite
\[
g:
\xymatrix{
UA 
\ar[r]^-{\eta^l U}
&
ULUA
\ar[r]^-{ULUf}
&
ULUL(C \ot ULX)
\ar[r]^-{U\eps'L}
&
ULX
}
\]
where $\eps'$ is the counit of the adjunction~\eqref{eq:adj}. As we shall use the  backwards expression, giving $f$ in terms of $g$, we need first to rewrite the unit $\eta'$: 
\[
\resizebox{\textwidth}{!}{$
\eta': \xymatrix@C=27pt{
1
\ar[r]^-{u \ot 1}
&
LC \ot (-)
\ar[r]^-{L_2 \ot 1}
\ar[dr]^{{\hol}^{-1}}
&
LC \ot L\I \ot (-)
\ar[r]^-{1 \ot {\hol}^{-1}}
&
LC \ot LU(-)
\ar[r]^-{{\hol}^{-1}}
\ar@{}[d]|{\eqref{h4l}}
&
L(C \ot ULU(-))
\ar@{=}[dl]
\\
&
&
L(C \ot U(-))
\ar[ur]^{L_2}
\ar@{}[u]|{\eqref{h6l}}
\ar[r]^{L(1 \ot \eta^l U)} 
&
L(C \ot ULU(-))
}$}
\]
with the purpose of emphasizing the composite 
\begin{equation}
\label{eq:eta-r}
\eta^r: 
\xymatrix{ 
1 
\ar[r]^-{u\ot 1} 
& 
LC\ot -  
\ar[r]^-{{\hol}^{-1}} 
& 
L(C \ot U(-))}
\end{equation}
as the unit for the claimed adjunction $U\dashv L(C \ot (-))$ and henceforth writing $\eta'$ as $\xymatrix@C=30pt@1{1 \ar[r]^-{\eta^r} & L(C \ot U(-)) \ar[r]^-{L(1 \ot \eta^l U)} & L(C \ot ULU(-))}$, in the usual manner of composing adjunctions. Now, the correspondence $g\mapsto f$ is given as expected by $f=L(1 \ot g)\circ \eta^r$, see the next diagram: 
\[
\resizebox{\textwidth}{!}{$
f:
\xymatrix@C=30pt{
A 
\ar[r]^-{\eta^r} 
& 
L(C \ot U(-))
\ar[r]^{L(1 \ot \eta^l U)}
\ar[dr]_{L(1 \ot g)}
&
L(C \ot ULUA)  
\ar[r]^-{L(1 \ot ULg)}
&
L(C \ot ULULX)
\ar[r]^{L(1 \ot U\eps^l L)}
&
L(C \ot ULX)
\\
&
&
L(C \ot ULX)
\ar[ur]_{L(1 \ot \eta^l UL)}
\ar@{-}@<-.25ex>`r[urr][urr]
\ar@{-}@<+.25ex>`r[urr][urr]
&
}
$}
\]

After some tedious computations using~\eqref{f-eq}, we see that $g$ equalises $UL\eta^l$ and $\eta^l UL$. Consequently, there is a unique arrow $\tilde g:UA \to X$ such that $\eta^l \circ \tilde g = g$:
\[
\xymatrix@C=60pt{
UA \ar[dr]^g \ar@{.>}[d]_{\tilde g}
\\
X 
\ar[r]^{\eta^l}
&
ULX 
\ar@<+.5ex>[r]^-{UL\eta^l}
\ar@<-.5ex>[r]_-{\eta^l UL}
& 
ULULX
}
\] 
Put now $\tilde f:=L(1 \ot \tilde g) \circ \eta^r$. It follows at once that $L(1 \ot \eta^l) \circ \tilde f = f$ holds. 
\[
\xymatrix@C=60pt{
A 
\ar[dr]^{f}
\ar@{.>}[d]_{\tilde f}
\\
L(C \ot X)
\ar[r]^-{L(1\ot \eta^l)}
&
L(C \ot ULX) 
\ar@<+.5ex>[r]^-{L(1 \ot UL\eta^l)}
\ar@<-.5ex>[r]_-{L(1 \ot \eta^l UL)}
& 
L(C \ot ULULX)
}
\]
The uniqueness of $\tilde f$ is now a consequence of $L(1 \ot \eta^l):L(C \ot X) \to L(C \ot ULX)$ being monomorphic. 
\end{proof}

%===================================================================%

\begin{remark}\label{rem:R}
\begin{enumerate}
\item As in the proof of the adjoint triangle theorem of~\cite{dubuc}, the counit of the adjunction $U \dashv L(C \ot (-))$ is obtain as the unique arrow into the equaliser~\eqref{corefl eq}:
\begin{equation}\label{eq:eps-r}
\xymatrix@C=50pt@R=25pt{
UL(C \ot X) 
\ar[r]^{UL(1 \ot \eta^l)} 
\ar@{.>}[d]_{\eps^r}
& 
UL(C \ot ULX)
\ar@<+.5ex>[r]^-{UL(1\ot UL\eta^l)}
\ar@<-.5ex>[r]_-{UL(1 \ot \eta^l UL)}
\ar[d]^{\theta L}
& 
UL(C \ot ULULX)
\ar[d]^{\theta LUL}
\\
X 
\ar[r]^{\eta^l} 
& 
ULX 
\ar@<+.5ex>[r]^-{UL\eta^l}
\ar@<-.5ex>[r]_-{\eta^l UL}
& 
ULULX
}
\end{equation}
\item The monoidal structure of $L(C \ot -)$, obtained by transporting the (co)monoidal structure of $U$ along the adjunction, can be simplified to the following 
\begin{equation}
\resizebox{.875\textwidth}{!}{%
$\begin{aligned}\label{eq:lax}
& \xymatrix{
r_0: \I 
\ar[r]^-{u} 
&
LC
}
\\
& \xymatrix@C=40pt{
r_2: L(C \ot X) \ot L(C \ot Y) 
\ar[r]^-{{\hol}^{-1}} 
&
L(C \ot X \ot UL(C\ot Y) )
\ar[r]^-{L(1\ot 1\ot \eps^r)}
& 
L(C \ot X \ot Y)
}
\end{aligned}
$}
\end{equation}
This can be seen from
\[
\xymatrix@C=35pt@R=13pt{
\I 
\ar[r]^{u}
&
LC \ar[r]^{{\hol}^{-1}}
\ar@<-.25ex>@{-}`d[dr][dr]
\ar@<+.25ex>@{-}`d[dr][dr]
\ar@{}[dr]|{\eqref{h1l}}
&
L(C\ot U\I)
\ar[d]^-{L(1 \ot u_0^{-1})}
\ar@{<-}`u[ll]`[ll]_{\eta^r}[ll]
\\
&
&
LC
}
\]
and from
\[\resizebox{.9\textwidth}{!}{%
\xymatrix@C=15pt@R=20pt{%1
L(C \ot X) \ot L(C \ot Y) 
\ar[r]^-{u \ot 1}
\ar@{=}[d]
&
LC \ot L(C \ot X) \ot L(C \ot Y)
\ar[r]^-{{\hol}^{-1}}
\ar[d]^{{\hol}^{-1}\ot 1}
\ar@{}[dr]|{\eqref{h2l}}
&
L(C \ot U(L(C \ot X) \ot L(C \ot Y)))
\ar[d]^-{L(1 \ot u_2^{-1})} 
\ar@{<-}`u[ll]`[ll]_{\eta^r}[ll]
\\%2
L(C \ot X) \ot L(C \ot Y)
\ar[r]^-{\eta^r \ot 1}
\ar@<-.25ex>@{-}`d[dr][dr]
\ar@<+.25ex>@{-}`d[dr][dr]
\ar@{}[dr]|{U\dashv L(C \ot -)}
&
L(C \ot UL(C \ot X)) \ot L(C \ot Y)
\ar[r]^{{\hol}^{-1}}
\ar[d]^{L(1 \ot \eps^r) \ot 1}
&
L(C \ot UL(C \ot X) \ot UL(C \ot Y))
\ar[dd]^{L(1 \ot \eps^r \ot \eps^r)}
\\%3
&
L(C \ot X) \ot L(C \ot Y)
\ar[d]^{{\hol}^{-1}}
\ar@{}[r]|{(M)+(N)}
&
\\%4
&
L(C \ot X \ot UL(C \ot Y))
\ar[r]^{L(1 \ot 1 \ot \eps^r)} 
&
L(C \ot X \ot Y)}}
\]

\item The left coHopf operator $\chol$ associated to the adjunction $U\dashv R$ is precisely ${\hol}^{-1}$, the inverse of the left Hopf operator, as shown in the diagram below:
\begin{equation}\label{eq:chol=hol(1-)}
\vcenter{\xymatrix@R=20pt@C=45pt{
\chol: L(C \ot X) \ot A 
\ar[r]^-{1 \ot \eta^r}
\ar[d]_{{\hol}^{-1}}
\ar@{}[dr]|{(N)}
&
L(C \ot X) \ot L(C \ot UA)
\ar[d]^{{\hol}^{-1}}
\ar`r[dd]`[dd]^{r_2}[dd]
\ar@{}[dr]|(.4){\eqref{eq:lax}}
&
\\%2
L(C \ot X \ot UA) 
\ar[r]^-{L(1 \ot \eta^r)}
\ar@{-}@<+.25ex>`d[dr][dr]
\ar@{-}@<-.25ex>`d[dr][dr]
\ar@{}[dr]|{U\dashv L(C \ot -)}
&
L(C \ot X \ot UL(C \ot UA))
\ar[d]^{L(1 \ot 1\ot \eps^r)}
&
\\
&
L(C \ot X \ot UA)
&
}}
\end{equation}
Consequently, the adjunction $U\dashv R$ is automatically left coHopf. 
\end{enumerate}
\end{remark}

%========================================================%

\subsection{Ambidextrous Hopf adjunctions}\label{ssec:ambidex}

We shall now turn to the special case when $C$ is (isomorphic to) the unit object $\I$.\linebreak 
Notice that having morphisms $u:\I \to L\I$, $ m:L\I \ot L\I \to L\I$ satisfying~\eqref{LC} implies that $u$ is a two-sided inverse for the multiplication-like morphism $ m$, such that the Frobenius laws~\eqref{eq:Frob_algebra} hold. By this means exactly that the comonoid $(L\I, L_2, L_0)$ is in fact a Frobenius monoid with unit $u$ and multiplication $ m$ (see Definition~\ref{def:frob}).

Condition~\ref{cond2} of Theorem~\ref{frob1} is then automatically satisfied, as $L\eta^l:LX \to LULX$ is always a split monomorphism. % from the adjunction triangles. 

We can therefore restate Theorem~\ref{frob1} as follows: 

\begin{theorem}\label{frob3}

Let $L\dashv U:\D \to \C $ a left Hopf adjunction between monoidal categories, such that the comonoid $(L\I,L_2,L_0)$ is a Frobenius monoid. If $L$ is precomonadic, then it is right adjoint to $U$ and the resulting adjunction $U\dashv L$ is left coHopf. 

\end{theorem}

Consequently, the functor $L$ not only carries a comonoidal structure, but also a monoidal one, as computed in Remark~\ref{rem:R}. The next proposition shows that $L$, endowed with these two structures, becomes Frobenius monoidal:

\begin{proposition}\label{prop:fmf}
Assume the hypotheses of Theorem~\ref{frob3} hold. Then $L$ is a Frobenius monoidal functor. 
\end{proposition}

\begin{proof}
As $L\dashv U$ is left Hopf and $U\dashv L$ is left coHopf, we can recover the left side of a linearly distributive functor (the left (co)strengths) as in Theorem~\ref{bihopf is linear}. That is, apply~\eqref{(co)strengths for Hopf} to obtain $\nll:LX\ot LY\to L(X \ot Y)$ and $\nlr:L(X \ot Y) \to LX\ot LY$. But from~\eqref{eq:lax} we see that $\nll$ coincides with $r_2$, and from~\eqref{eq:chol=hol(1-)} and \eqref{h4l}, $\nlr$ is precisely $L_2$. 

Therefore the proof of the Frobenius law~\eqref{eq1:Frob} reduces to~\eqref{eq:proof of lf4}, and 
the proof of the second Frobenius law~\eqref{eq2:Frob} to~\eqref{eq:proof of lf5}. 
We include them below only for completeness.

Proof of~\eqref{eq1:Frob}:
\[
\resizebox{\textwidth}{!}{
\xymatrix@C=33pt@R=20pt{%1
&
LX \ot L(Y \ot Z) 
\ar[r]^{{\hol}^{-1}}
\ar[d]^{1 \ot L(1 \ot \eta^l)}
\ar@<+.1ex>`l[dd]`[dd]_{1 \ot L_2} [dd]
\ar@{}[dr]|{(N)}
&
L(X \ot UL(Y \ot Z))
\ar[d]^{L(1 \ot UL(1 \ot \eta^l))}
\ar[r]^-{L(1 \ot \eps^r)}
\ar@{}[dr]^(.6){(N)}
&
L(X \ot Y \ot Z)
\ar[d]^{L(1 \ot 1\ot \eta^l)}
\ar@{<-}`u[ll]`[ll]_{r_2}[ll]
\\%2
&
LX \ot L(Y \ot ULZ)
\ar[r]^{{\hol}^{-1}}
\ar[d]^{1 \ot \hol}
\ar@{}[dr]|{(N)}
\ar@{}[d]_{\eqref{h4l}}
&
L(X \ot UL(Y \ot ULZ))
\ar[r]^-{L(1 \ot \eps^r)}
\ar[d]|{L(1 \ot U\hol) \overset{\eqref{eq:chol=hol(1-)}}{=} L(1 \ot U{\chol}^{-1})}
&
L(X \ot Y \ot ULZ)
\ar@{=}[dd]
\\%3
& 
LX \ot LY \ot LZ
\ar[r]^{{\hol}^{-1}}
\ar[d]^{{\hol}^{-1} \ot 1}
\ar@<+.4ex>`l[dddrr]`[dddrr]_{r_2\ot 1}`[ddrr][ddrr]
\ar@{}[dr]|{\eqref{h2l}}
&
L(X \ot U(LY \ot LZ))
\ar[d]^{L(1 \ot u_2^{-1})}
\ar@{}[r]|(.65){\mathsf{co}\eqref{h5l}}
\ar@{}[r]|(1.2){\eqref{h4l}}
&
\\%4
&
L(X \ot ULY) \ot LZ 
\ar[r]^{{\hol}^{-1}}
\ar@{-}@<-.25ex>`d[dr][dr]
\ar@{-}@<+.25ex>`d[dr][dr]
&
L(X \ot ULY \ot ULZ)
\ar[r]^-{L(1 \ot \eps^r \ot 1)}
\ar[d]^{\hol}
\ar@{}[dr]^(.6){(N)}
&
L(X \ot Y \ot ULZ)
\ar[d]^{\hol}
\\%5
&
&
L(X \ot ULY) \ot LZ 
\ar[r]^-{L(1 \ot \eps^r) \ot 1}
&
L(X \ot Y) \ot LZ
\ar@{<-}`r[uuuu]`[uuuu]_{L_2}[uuuu]
\\
&
&
&
}}
\]

Proof of~\eqref{eq2:Frob}:
\[
\xymatrix@R=20pt@C=55pt{%1
L(X \ot Y) \ot LZ
\ar[r]^{L_2 \ot 1}
\ar[d]_{{\hol}^{-1}}
&
LX \ot LY \ot LZ 
\ar[d]^{1 \ot {\hol}^{-1}}
\\%2
L(X \ot Y \ot ULZ)
\ar[r]^{L_2}
\ar[d]|{L(1 \ot 1 \ot \eps^r)}
\ar@{}[dr]|{(N)}
\ar@{}[ur]|{\eqref{h6l}}
&
LX \ot L(Y \ot ULZ)
\ar[d]|{1 \ot L(1 \ot \eps^r)}
\\%3
L(X \ot Y \ot Z) 
\ar[r]^{L_2}
\ar@{<-}`l[uu]`[uu]^{r_2}[uu]
&
LX \ot L(Y \ot Z) 
\ar@{<-}`r[uu]`[uu]_{1 \ot r_2}[uu]
}
\]
\end{proof}

\begin{remark}\label{rem:biHopf}
If in Theorem~\ref{frob3}, the adjunction $L\dashv U$ is assumed to be \emph{right} Hopf instead of left Hopf, similar reasoning produces again the adjunction $U\dashv L$, but this time \emph{right coHopf}, with unit 
\[
{\bar\eta}^r : \xymatrix{
1 
\ar[r]^-{1 \ot u}
& 
- \ot L\I 
\ar[r]^-{{\hor}^{-1}}
& 
LU(-)
}
\]
and counit ${\bar\eps}^r :UL \to 1$ the unique natural transformation such that
\[
\xymatrix{
ULX 
\ar[r]^-{UL\eta^l}
\ar[d]_{{\bar\eps}^r}
&
ULULX
\ar[d]^{\bar \theta L}
\\
X 
\ar[r]^-{\eta^l}
&
ULX
}
\] 
holds, where $\bar\theta$ is the composite\footnote{Notice that the evaluation and coevaluation morphisms for the dual pair $(L\I, L\I)$ do not change.}  
\[
\resizebox{\textwidth}{!}{
\xymatrix{
\bar\theta: 
ULU(-)
\ar[r]^-{\eta^l ULU} 
& 
ULULU(-)
\ar[r]^-{U\hor} 
& 
U(LU(-) \ot L\I)
\ar[r]^-{U(\hor)\ot 1}
& 
U((-) \ot L\I\ot L\I)
\ar[r]^-{U(-\ot \e)} 
& 
U
}}
\]
Now the monoidal structure on $L$ is 
\begin{align*}
&
{\bar r}_0: \xymatrix@C=40pt{
\I 
\ar[r]^u
&
L\I}
\\
&
{\bar r}_2: \xymatrix@C=40pt{
LX \ot LY 
\ar[r]^-{{\hor}^{-1}} 
&
L(ULX \ot Y) 
\ar[r]^-{L({\bar \eps}^r \ot 1)}
&
L(X \ot Y)
}
\end{align*}

If we go one step further with the assumptions on Theorem~\ref{frob3}, and ask for the original adjunction $L\dashv U$ to be both left and right Hopf, we obtain that $L$ is right adjoint to $U$ in two possibly different%
\footnote{These two structures are generally not the same; consider for example a finite dimensional Hopf algebra $H$ over a commutative field, such that $H$ is not \emph{unimodular} (like Sweedler's four dimensional Hopf algebra), and the Hopf adjunction between free functor $L=-\ot H$ and the forgetful functor $U$ from the category of right $H$-modules to the category of vector spaces. Then the two adjunctions $U\dashv L$ have different (co)units, therefore, they are not identical.
} %
ways, one being left coHopf with unit $\eta^r$ and counit $\eps^r$ given by~\eqref{eq:eta-r} and~\eqref{eq:eps-r}, and the other right coHopf with ${\bar \eta}^r$, ${\bar \eps}^r$ described above. As the standard composite $\xymatrix@1{L \ar[r]^-{\eta^r L} & LUL \ar[r]^-{L {\bar \eps}^r} & L}$ is a (monoidal) isomorphism between the two versions of right adjoints of $U$, it follows that $(L\dashv U, \eta^r, \eps^r)$ is not only left coHopf, but also right coHopf. 

Therefore a \emph{two-sided} Hopf adjunction $L\dashv U$ such that $L$ is comonadic and $L\I$ is a Frobenius monoid yields a \emph{two-sided} biHopf and \emph{ambidextrous} adjunction $L\dashv U\dashv L$, with $L$ Frobenius monoidal. 

\end{remark}

\begin{remark}

Compare the above results with~\cite{sh}, where a very particular case of a Hopf adjunction between autonomous categories was considered.

More precisely, let $\C$ denote a finite tensor category. Then the forgetful functor $U$ from the monoidal center of $\C$ is strong monoidal, and has both left adjoint~\cite{ds:centres} and right adjoint~\cite[Proposition~7.6]{day-street:quantum}; in particular it induces a biHopf adjunction $L\dashv U\dashv R$. In~\cite{sh}, it is shown that the left and right adjoints of $U$ are isomorphic, $L\cong R$, if and only if $L$ commutes with duals (on the level of objects only), if and only if $L{\I}$ is a self dual object (under the additional assumption on $\I$ of being simple), if and only if $\C$ is unimodular - the latter being a specific property of finite tensor categories.%, corresponding to our assumption $C\cong \I$. 

Our results were obtained in a more general setting of \emph{arbitrary} Hopf adjunctions between monoidal categories, but with the additional hypotheses that $L$ is precomonadic and that $L\I$ is not only a self-dual object, but a Frobenius monoid. The exactness conditions we impose are fulfilled in~\cite{sh} because of the special properties of the category considered. Consequently, the results we obtain are more precise, in the sense that, first, we provide conditions for the existence of the right adjoint, and second, that the resulting ambidextrous adjunction $L\dashv U\dashv L$ takes into account also the (co)monoidal nature of $L$, making it a Frobenius monoidal functor;%
\footnote{%
It is known that Frobenius monoidal functors preserve dual pairs~\cite{day-pastro:Frob}. In~\cite{balan:frob}, it was shown conversely that a functor between autonomous categories which preserves duals \emph{in a coherent manner with respect to evaluation and coevaluation morphisms}, is Frobenius monoidal.} while in~\cite{sh} this aspect does not seem to be fully considered. 

\end{remark}

%========================================================%

\section{Hopf monads and Frobenius-type properties} \label{sec:hm}

%===================================================================%

\subsection{On Hopf (co)monads}\label{sect:hm}

A monad $(T, \mu:T^2\to T, \eta:1 \to T)$ on a monoidal category $\C$ is called \emph{a comonoidal monad} if its functor part is endowed with a comonoidal structure \linebreak $(T_2: T(X \ot Y) \to TX \ot TY, T_0:T\I \to \I)$, such that the unit $\eta$ and the multiplication $\mu$ become comonoidal natural transformations. 

A comonoidal monad is \emph{a left Hopf monad}~\cite{blv:hmmc} if the following natural transformation (called left fusion operator) is an isomorphism: 
\[
\xymatrix@C=15pt{T(- \ot T(-) ) \ar[r]^-{T_2} & T(-) \ot T^2(-) \ar[r]^{1 \ot \mu }& T(-) \ot T(-)}
\]
A right Hopf monad~\cite{blv:hmmc} is a left Hopf monad on $\C^\cop$. 

\medskip

Any adjunction $L\dashv U:\D \to \C$ produces a monad $T=UL$. If the adjunction is comonoidal, so is the monad. Furthermore, if the adjunction is left or right Hopf, the monad is again such~\cite{blv:hmmc}. 

Conversely, for any monad $T$ on a category $\C$, denote as in (the end of) Section~\ref{ssec:frob mon} by $\C^T$ the Eilenberg-Moore category of $T$-algebras, and by $L^T\dashv U^T:\C^T \to \C$ the adjunction between the free and the forgetful functor. There is a one-to-one correspondence between comonoidal structures on the underlying functor $T$ making it a comonoidal monad, 
and monoidal structures on the category $\C^T$ such that the forgetful functor $U^T$ becomes strict monoidal~\cite{moerdijk}, in particular exhibiting $L^T \dashv U^T$ as a comonoidal adjunction. One step further, $L^T \dashv U^T$ is a left/right Hopf adjunction if and only if $T$ is a left/right Hopf monad~\cite{blv:hmmc}.%

%===================================================================%

\paragraph{Hopf comonads} By reversing the arrows, one obtains the notions of a monoidal comonad $(G:\C \to \C,g_2:GX \ot GY \to G(X \ot Y),g_0:\I \to G\I)$, respectively a left/right Hopf comonad~\cite{cls:hmc} on a monoidal category. 

The comonad $G=UR$ generated by a coHopf adjunction $U \dashv R$ is a Hopf comonad and conversely, any left/right Hopf comonad $G$ on a monoidal category $\C$ produces a left/right coHopf adjunction $U_G \dashv R_G:\C \to \C_G$ between the corresponding forgetful and cofree functors, where $\C_G$ denotes the Eilenberg-Moore category of $G$-coalgebras. The reader can easily derive their descriptions by duality.

%========================================================%

\subsection{BiHopf monads produce linearly distributive \\ (co)monads}\label{hm and lf}

Recall that (co)Hopf adjunctions determine Hopf (co)monads and conversely, any Hopf (co)monad induces a (co)Hopf adjunction between the (co)free and forgetful functor. Any biHopf adjunction $L\dashv U\dashv R$ determines not only an adjunction between the Hopf monad $T=UL$ and the Hopf comonad $G=UR$, but also a linearly distributive adjunction $(U,U)\dashv (R,L)$. Consequently, the composite $(G,T)$ is a comonad in the 2-category $\mathbf{MonLinDist}$.

Conversely, any monad $T$ on a category $\C$, having a \emph{right adjoint} $G$, determines a comonad structure on $G$. The categories of $G$-coalgebras $\C_G$ and of $T$-algebras $\C^T$ are isomorphic~\cite{em:adjoint}, such that the isomorphism functor $Q:\C^T \to \C_G$ commutes with the forgetful functors, namely $U_G =U^T Q$. Via the isomorphism $Q$, the forgetful functor $U^T$ has not only a left adjoint $L^T$, but gains also a right adjoint $R^T:=Q R_G$.

If we assume that $T$ is a comonoidal monad, then its right adjoint $G$ becomes a monoidal comonad, while the forgetful functors $U^T$ and $U_G$ are strong monoidal and so is $Q$. Thus in the triple adjunction $L^T \dashv U^T \dashv R^T$, the functor $U^T$ is strong monoidal, while its left and right adjoints are comonoidal, respectively monoidal. 

One step further, suppose that $T$ is a Hopf monad and its adjoint $G$ is a Hopf comonad. We shall call such $T$ a \emph{biHopf monad}. For a biHopf monad $T$, with right adjoint $G$, the adjunction $L^T \dashv U^T$ is a Hopf adjunction, and $U_G \dashv R_G$ is a coHopf adjunction. As (co)Hopf adjunctions are stable under strong monoidal isomorphisms, $U^T \dashv R^T$ is a coHopf adjunction. Consequently, $L^T \dashv U^T \dashv R^T$ is a biHopf adjunction, thus Corollary~\ref{bihopf} applies and yields the comonad $(G,T)$ in the 2-category $\mathbf{MonLinDist}$. We record the above discussion in the following:

\begin{corollary}

Let $T$ a biHopf monad on a monoidal category $\C$. Then the pair $(G, T)$ is a comonad in $\mathbf{MonLinDist}$, induced by the adjunction $(U^T, U^T)\dashv (R^T, L^T)$.

\end{corollary}

\begin{remark}
A Hopf monad $T$ on a monoidal category which has a right adjoint is not necessarily a biHopf monad. %\footnote{Closed versus coclosed } 

To see this, let $\C$ be the category of sets and functions $\mathsf{Set}$, with its usual cartesian monoidal structure, and let $\mathfrak G$ be a (finite) non-trivial group. Then the monad $T= - \times \mathfrak G $ is easily seen to be both left and right Hopf, while its right adjoint (comonad) $G= \mathsf {Set}(\mathfrak G,-)$ is not a Hopf comonad - the cofusion operators failing to be bijective for cardinality reasons.

However, in case the base category is \emph{autonomous}, $T$ is a Hopf monad if and only if it is a biHopf monad. %
This happens because: (i) the associated category of $T$-algebras is autonomous~\cite{bv:hm}, and (ii) the forgetful functor $U^T$ is strong (strict, in fact) monoidal. %
Then by~\cite[Proposition~7.6]{day-street:quantum}, the functor $U^T$ has not only as left adjoint the free functor $L^T$, but also gains a right adjoint $R^T$ given by the double dual of $L^T$.\footnote{%
There is however an error in \emph{op.~cit.}, in the sense that it asserts creation of an infinite chain of adjunction. In fact one can only prove that $U$ has left adjoint if and only if it has a right adjoint.} %
Consequently, there is a triple {\em biHopf} adjunction $L^T \dashv U^T \dashv R^T$, due to $U^T$ {\em being strong monoidal between autonomous categories}. Thus $T$ has the structure of biHopf monad (see also the more recent~\cite[Corollary~3.12]{bv:hm} and \cite[Lemma~3.5]{bv:qd} on the existence of the right adjoint of $T$). 

\end{remark}

%===================================================================%

\subsection{Hopf monads are Frobenius}\label{ssec:hm are fm}

Having in mind the results established in Section~\ref{sec:main}, we turn now to the corresponding statements for Hopf monads. We refer to the notations introduced in Sections~\ref{sect:hm} and~\ref{hm and lf}.

Let $T$ a left Hopf monad on a monoidal category $\C$. Then Theorem~\ref{frob1} rephrases as follows:

\begin{theorem}\label{frob4}

Let $T$ be a left Hopf monad on a monoidal category $\C$. Assume that there is an object $C\in \C $, together with $T$-algebra morphisms $\overline u: \I \to TC$, \linebreak $\overline  m: T\I \ot TC \to T\I$, satisfying the analogues of~\eqref{LC}. Suppose that for any $X\in \C$, the diagram below is an equaliser\footnote{Equivalently, that the monad $T$ is of \emph{descent type}.}
\begin{equation}\label{monad corefl eq}
\xymatrix@R=8pt@C=45pt{X \ar[r]^{\eta} & TX \ar@<0.5ex>[r]^{T\eta} \ar@<-0.5ex>[r]_{\eta T} & T^2X }
\end{equation}
and that $T(1 \ot \eta):T(C \ot -) \to T(C \ot T(-))$ is a component-wise monomorphism. 
Then $T(C \ot -)$ is right adjoint to $T$.
\end{theorem}

\begin{proof}

As in Section~\ref{ssec:wirth}, we can see that the object $L^T C =(TC, \mu)$ is right dual to $L^T \I = (T\I, \mu)$ (using that both $\overline u$ and $\overline  m$ are $T$-algebra morphisms). 

Notice that the equaliser~\eqref{monad corefl eq} is precisely~\eqref{corefl eq} for the left Hopf adjunction \linebreak$L^T \dashv U^T$, and that $L^T(1 \ot \eta)$ is component-wise monomorphic if and only if $T(1 \ot \eta)$ is so. Then theorem~\ref{frob1} applies to give $U^T \dashv L^T(C \ot (-))$. In particular we obtain the desired adjunction 
\[
T=U^TL^T \dashv U^TL^T(C \ot (-))=T(C \ot (-))
\]  
\end{proof}

\begin{remark}
\begin{enumerate}

\item Let $\overline\e$ and $\overline\db$ be the evaluation, respectively the coevaluation morphisms of the dual pair $(T\I, TC)$ as in~\eqref{LC-right-dual-LI}. Recall that to give $\overline u$ and $\overline  m$ satisfying~\eqref{LC} is the same as to exhibit $TC$ as right dual for $T\I$, such that $T_2:TC \to TC \ot T\I$ is the twice transpose of $T_2:T\I \to T\I \ot T\I$. Having then $\overline u$ and $\overline  m$ morphisms of $T$-algebras is the same as requiring that $\overline\e$ and $\overline\db$ are morphism of $T$-algebras; that is, $(T\I, TC)$ is a dual pair in $\C^T$, where both are seen as free $T$-algebras.  

\item Let us see what happens in case we only require $T\I$ and $TC$ to form a dual pair in $\C$. First, there is an adjunction $-\ot TC \dashv -\ot T\I$. 
The monad $T$ being comonoidal, and $TC$ carrying the structure of a free $T$-algebra, the functor $- \ot TC$ lifts to $T$-algebras via the distributive law given by the left fusion operator $\bhol: T(- \ot TC) \to T(-) \ot TC$. 
But $T$ is a left Hopf monad,%
\footnote{There is another adjunction we might had considered, namely $T\I \ot- \dashv TC \ot -$. Both functors obviously lift to $T$-algebras. But in order to lift their adjunction, the distributive law corresponding to the left adjoint, namely $\bhor: T(T\I \ot -) \to T\I \ot T(-)$, should be invertible. That is, the monad $T$ should be \emph{right} (pre)Hopf monad, instead of left Hopf.} 
hence the fusion operator $\bhol$ is invertible. So not only $-\ot TC$ but the whole adjunction $-\ot TC \dashv -\ot T\I$ will lift to $\C^T$ (see for example~\cite{johnstone:adjoint-lift}).%
\footnote{The functor $(-) \ot T\I$ does have a lift to $T$-algebras using again the left fusion operator as a distributive law, but this does not mean that the two liftings are the same. After all, a functor might have more than one lift to $T$-algebras (see for example~\cite[Remark~2.3]{bk:coalg-alg}).} %
How does the lifted right adjoint behave like? If we unravel its construction, we see that it maps a $T$-algebra $(X, x:TX \to X)$ to $X \ot T\I$, with $T$-algebra structure:
\[
\resizebox{.95\textwidth}{!}{%
\xymatrix@C=25pt{
T(X \ot T\I) 
\ar[r]^-{1 \ot \overline \db}
&
T(X \ot T\I) \ot TC \ot T\I
\ar[r]^-{{\bhol}^{-1} \ot 1}
& 
T(X \ot T\I \ot TC) \ot T\I
\ar[r]^-{T(1 \ot \overline \e) \ot 1}
&
TX \ot T\I
\ar[r]^-{x \ot 1}
& 
X \ot T\I
}}
\]
where $\overline\e$ and $\overline\db$ as in~\eqref{LC-right-dual-LI}. Instantiating the above at the unit object, we obtain a $T$-algebra structure on $T\I$, providing the left dual of $L^T C=(TC,\mu)$, such that $\overline\e$ and $\overline\db$ become morphisms of $T$-algebras:
\[
\resizebox{.95\textwidth}{!}{%
\xymatrix@C=25pt{
T^2\I 
\ar[r]^-{1 \ot \overline \db}
&
T^2\I \ot TC \ot T\I
\ar[r]^-{{\bhol}^{-1} \ot 1}
& 
T(T\I \ot TC) \ot T\I
\ar[r]^-{T\overline \e \ot 1}
&
T\I \ot T\I
\ar[r]^-{T_0 \ot 1}
& 
T\I
}}
\]
But is this the free $T$-algebra structure $\mu:T^2\I \to T\I$? 
The computations below show that this is indeed the case if $\overline m$ is a morphism of $T$-algebras (the diagram marked ($*$)): 
\[
\resizebox{.975\textwidth}{!}{%
\xymatrix@R=20pt{%1
T^2\I 
\ar[r]^-{1 \ot \overline u}
\ar[ddd]_{\mu}
\ar@{}[dddrrr]|{(M)}
&
T^2\I \ot TC
\ar[r]^-{1 \ot T_2}
&
T^2\I \ot TC \ot T\I
\ar[r]^-{{\bhol}^{-1}\ot 1}
\ar@{-}@<+.25ex>`d[ddr][ddr]
\ar@{-}@<-.25ex>`d[ddr][ddr]
&
T(T\I \ot TC) \ot T\I
\ar[r]^-{T\overline  m \ot 1}
\ar[d]^{T_2 \ot 1}
&
T^2\I \ot T\I
\ar[ddd]_{\mu \ot 1}
\ar[r]^-{TT_0 \ot 1}
&
T\I \ot T\I
\ar[dddd]^{T_0 \ot 1}
\\%2
&
&
&
T^2\I \ot T^2C \ot T\I
\ar[d]^{1 \ot \mu \ot 1}
\ar@{}[dr]|{(*)}
&
&
\\%3
&
&
&
T^2\I \ot TC \ot T\I
\ar[d]_{\mu \ot 1\ot 1}
&
&
\\%4
T\I
\ar[rr]^{1 \ot \overline  u}
\ar@{-}@<-.25ex>`d[drrr][drrr]
\ar@{-}@<+.25ex>`d[drrr][drrr]
&
&
T\I \ot TC
\ar[r]^-{1 \ot T_2}
\ar[dr]^{\overline  m}
\ar@{}[dl]|-{\hyperref[LC]{\eqref{LC}\text{-(iii)}}}
&
T\I \ot TC \ot T\I 
\ar[r]^{\overline  m \ot 1}
\ar@{}[d]|-{\hyperref[LC]{\eqref{LC}\text{-(ii)}}}
&
T\I \ot T\I
\ar[dr]^{T_0 \ot 1}
\\%5
&
&
&
T\I 
\ar[ur]^{T_2}
\ar@{=}[rr]
&
&
T\I
}}
\]
Notice also that $\overline  u$ is the transpose of $T_0$ under duality:
\[
\overline u = \xymatrix{ \I \ar[r]^-{\overline\db} & TC \ot T\I \ar[r]^-{1 \ot T_0} & TC}
\]
As the latter is indeed a morphism of $T$-algebras, so it will be $\overline  u$. Consequently, in Theorem~\ref{frob4}, the hypothesis that $\overline u$ is a morphism of $T$-algebras is redundant. 

\end{enumerate}
\end{remark}

Particularizing now Theorem~\ref{frob4} to the case $C \cong \I$ and using Theorem~\ref{frob3}, Proposition~\ref{prop:fmf} and Remark~\ref{rem:biHopf}, we obtain the following:

\begin{theorem}\label{frob5}

Let $T$ be a Hopf monad on a monoidal category $\C$, such that the free $T$-algebra $T\I$ is a Frobenius monoid in $\C^T$,\footnote{Consequently, $T\I$ will also be a Frobenius monoid in $\C$.} with comonoid structure given by $T_2, T_0$, and that $T$ is of descent type. 

Then $T$ becomes right adjoint to itself; in particular, $T$ is not only a Hopf monad but also a biHopf monad, a Frobenius monad and a Frobenius monoidal functor.

\end{theorem}

We end the section by an example which shows that not any Hopf monad is Frobenius monoidal, nor a Frobenius monad.

\begin{example} Let $\Bbbk$ be a commutative field of characteristic different from $2$, and consider the $\mathbbm Z_2$-graded Hopf algebra $A=\Bbbk [x]/(x^2)$, with grading $A_0=\Bbbk$ and $A_1$ spanned by $x$. It is a monoid as usual, with the induced (polynomial) multiplication and unit $1$. As for the comonoid structure, the element $1$ is group-like and $x$ is primitive. The antipode acts as $S(1)=1, S(x)=-x$. This is a Hopf algebra in the braided category $\C$ of (finite dimensional) $\mathbbm Z_2$-graded vector spaces, which is not graded Frobenius (that is, it is not a Frobenius monoid in $\C$), but still a Frobenius monoid in the category of vector spaces~\cite{dnn:frob}. 

Then tensoring with $A$ yields the Hopf monad $T=A \ot -$ on the (braided) autonomous category $\C$. But $T$ fails to be a Frobenius monad, and also a Frobenius monoidal functor, because in $\C$, the Hopf algebra $T\I = A$ is not a Frobenius monoid.

\end{example}

\end{document}